\documentclass{amsart}
\usepackage{amsthm}
\usepackage{mathtools}
\usepackage{xcolor}
\usepackage{amssymb}
\usepackage{amsmath}
\usepackage{tabularx}
\usepackage{graphicx}
\usepackage{float}
\usepackage{placeins}
\usepackage[all]{xy}
\usepackage{todonotes}
\usepackage{tikz}
\usepackage{mathpazo}
\usepackage{lscape}
\usepackage{enumerate}
\usepackage{subcaption}
\usepackage[pagebackref,colorlinks=true,
	urlcolor=blue!20!black,
	citecolor=green!50!black]{hyperref}
\usepackage[nobysame]{amsrefs}
\makeatletter
\let\orgdescriptionlabel\descriptionlabel
\renewcommand*{\descriptionlabel}[1]{%
  \let\orglabel\label
  \let\label\@gobble
  \phantomsection
  \protected@edef\@currentlabel{#1\unskip}%
  \let\label\orglabel
  \orgdescriptionlabel{(#1)}%
}
\makeatother
\theoremstyle{plain}
\newtheorem{theorem}{Theorem}[section]
\newtheorem{lemma}[theorem]{Lemma}
\newtheorem{proposition}[theorem]{Proposition}
\newtheorem{corollary}[theorem]{Corollary}

\newtheorem{question}[theorem]{Question}

\newtheorem{conjecture}[theorem]{Conjecture}
\theoremstyle{remark}

\newtheorem{example}[theorem]{Example}
\newtheorem{remark}[theorem]{Remark}

\newcommand{\defn}[1]{{\color{green!50!black}\emph{#1}}}
\newcommand{\ie}{\text{i.e.}\;}
\newcommand{\defs}{\overset{\mathsf{def}}{=}}
\newcommand{\bubcov}{\lessdot_{\mathsf{bub}}}
\newcommand{\bubless}{<_{\mathsf{bub}}}
\newcommand{\bubleq}{\leq_{\mathsf{bub}}}
\newcommand{\indel}{\rightarrow}

\newcommand{\shufleq}{\leq_{\mathsf{shuf}}}
\newcommand{\compleq}{\leq_{\mathsf{comp}}}
\newcommand{\compcov}{\lessdot_{\mathsf{comp}}}
\newcommand{\sindel}{\hookrightarrow}
\newcommand{\transpose}{\Rightarrow}
\newcommand{\Poset}{\mathbf{P}}
\newcommand{\Lattice}{\mathbf{L}}
\newcommand{\MI}{\mathsf{M}}
\newcommand{\JI}{\mathsf{J}}

\newcommand{\Galois}{\mathsf{Galois}}	
\newcommand{\Covers}{\mathsf{Edge}}
\newcommand{\jsdlabeling}{\lambda_{\mathsf{jsd}}}
\newcommand{\canset}{\mathsf{Can}}

\newcommand{\invset}{\mathsf{Inv}}
\newcommand{\Shuf}{\mathsf{Shuf}}
\newcommand{\Tri}{\mathsf{Tri}}
\newcommand{\ShufPoset}{\textbf{\textsf{Shuf}}}
\newcommand{\Bub}{\textbf{\textsf{Bub}}}
\newcommand{\Hoch}{\textbf{\textsf{Hoch}}}

\def\abf{\mathbf{a}}\def\bbf{\mathbf{b}}\def\ubf{\mathbf{u}}\def\vbf{\mathbf{v}}\def\wbf{\mathbf{w}}\def\xbf{\mathbf{x}}\def\ybf{\mathbf{y}}

\def\Tcal{\mathcal{T}}

\def\rd{\textcolor{red}}
\def\bl{\textcolor{blue}}

\title{Bubble Lattices I: Structure}
\author[T.~McConville]{Thomas McConville}
\address[T.~McConville]{Kennesaw State University, Department of Mathematics, 30144 Kennesaw (GA), USA}
\email{tmcconvi@kennesaw.edu}
\author[H.~M\"uhle]{Henri M{\"u}hle}
\address[H.~M\"uhle]{Technische Universit{\"a}t Dresden, Institut f{\"u}r Algebra, Zellescher Weg 12--14, 01069 Dresden, Germany.}
\email{henri.muehle@tu-dresden.de}

\begin{document}

\allowdisplaybreaks

\begin{abstract}
	C.~Greene introduced the shuffle lattice as an idealized model for DNA mutation and discovered remarkable combinatorial and enumerative properties of this structure. We attempt an explanation of these properties from a lattice-theoretic point of view.  To that end, we introduce and study an order extension of the shuffle lattice, the \emph{bubble lattice}.  We characterize the bubble lattice both locally (via certain transformations of shuffle words) and globally (using a notion of inversion set).  We then prove that the bubble lattice is extremal and constructable by interval doublings.  Lastly, we prove that our bubble lattice is a generalization of the Hochschild lattice studied earlier by Chapoton, Combe and the second author.  
\end{abstract}

\maketitle

\section{Introduction}

Motivated by an idealized model for mutations in DNA sequences, C.~Greene introduced the \emph{shuffle lattice}.  The ground set of this lattice is the set $\Shuf(m,n)$ of shuffles of order-preserving repetition-free words whose letters are taken from two disjoint, linearly ordered alphabets $X=\{x_{1},x_{2},\ldots,x_{m}\}$ and $Y=\{y_{1},y_{2},\ldots,y_{n}\}$.  The \emph{shuffle order} is determined by inserting letters of $Y$ or deleting letters of $X$ from any given \emph{shuffle word}. Therefore, the maximal chains in the resulting \emph{shuffle lattice} $\ShufPoset(m,n)$ describe the possible ways of transforming $x_1x_2\cdots x_m$ to $y_1y_2\cdots y_n$ by changing one letter at a time.

In \cite{greene:shuffle}, Greene studied this poset extensively and discovered several surprising enumerative relations among its characteristic polynomial, its zeta polynomial and its rank-generating function.  Namely, each of these invariants occurs as a specialization of the same Jacobi polynomial. Greene's enumerative results were recovered in \cite{simion_stanley:shuffle} using algebraic methods, but the presence of the Jacobi polynomials remained mysterious.  Various extensions of the shuffle lattice have been studied in \cites{hersh:shuffle,doran:shuffling}.

This article is the first of two papers studying an order extension of the shuffle order, where we additionally allow exchanges of adjacent letters from $X$ and $Y$.  We call this order extension the \emph{bubble order}, because the exchange operation is slightly reminiscient of the bubble sort algorithm.  We exhibit several interesting enumerative and structural connections between shuffle and bubble lattices using combinatorial lattice theory, which may shed some light on the enumerative relationship among the combinatorial invariants of the shuffle lattice.  More precisely, in the present article we study order- and lattice-theoretic properties of this new family of lattices.  In the second article~\cite{mcconville.muehle:bubbleII} we study certain enumerative invariants of these lattices and investigate their geometric structure.
%

To our knowledge, the lattices $\Bub(m,n)$ have not appeared in this generality in the literature before.  
Figure~\ref{fig:bubble_shuffle} shows Greene's shuffle lattice $\ShufPoset(2,1)$ and our bubble lattice $\Bub(2,1)$ side by side.

\begin{figure}
	\centering
	\begin{subfigure}[t]{.45\textwidth}
		\centering
		\includegraphics[page=4,scale=1]{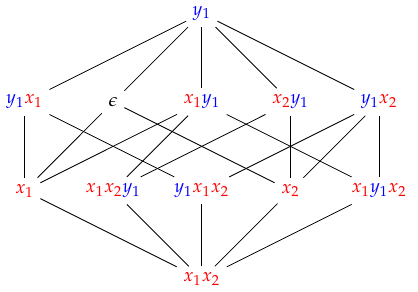}
		\caption{The shuffle poset $\ShufPoset(2,1)$.}
		\label{fig:shuffle_21}
	\end{subfigure}
	\hspace*{1cm}
	\begin{subfigure}[t]{.45\textwidth}
		\centering
		\includegraphics[page=5,scale=1]{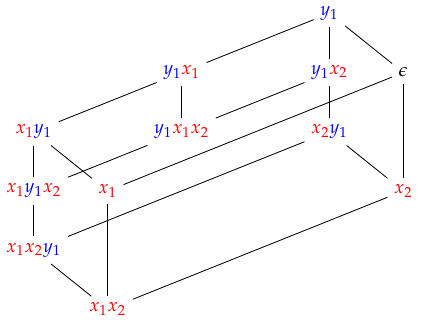}
		\caption{The bubble poset $\Bub(2,1)$.}
		\label{fig:bubble_21}
	\end{subfigure}
	\caption{Two posets of shuffle words.}
	\label{fig:bubble_shuffle}
\end{figure}

Our main result states that the bubble lattice is indeed a lattice.  On top of that we prove that it enjoys several remarkable combinatorial properties which we explain in detail in Section~\ref{sec:lattices}.  The following theorem is proven over the course of Section~\ref{sec:bubble_lattice}.

\begin{theorem}\label{thm:bubble_lattice_main}
	For $m,n\geq 0$, the bubble lattice $\Bub(m,n)$ is a lattice.  Moreover, $\Bub(m,n)$ is extremal, semidistributive and constructable by interval doublings.
\end{theorem}


%

In Section~\ref{sec:hochschild}, we show that $\Bub(m,1)$ is isomorphic to a so-called \emph{Hochschild lattice}, studied for instance in \cites{combe:geometric,muehle:hochschild,chapoton:dyck}.  In fact, the research presented in this article was motivated by the results of \cite{muehle:hochschild}, where it was shown that $\ShufPoset(m,1)$ arises from $\Bub(m,1)$ (treated in a different guise) through a combinatorial construction known as the \emph{core label order}, see \cite{muehle:core}.  This construction has its origins in N.~Reading's \emph{shard intersection order} in the context of posets of regions of hyperplane arrangements~\cite{reading:lattice_theory} and has a huge significance in the field of Coxeter--Catalan combinatorics~\cites{reading:shard,reading:clusters}.  Our initial goal was to exhibit a lattice $\Lattice_{m,n}$ whose core label order is isomorphic to the shuffle lattice $\ShufPoset(m,n)$, thus addressing \cite{muehle:hochschild}*{Question~7.1}.  The bubble lattices do not quite have this property.  Our data suggest they are close in the sense that their core label orders share the rank enumeration with the bubble lattices, but in general have more covering pairs.  Moreover, in general, the core label orders of the bubble lattices are not themselves lattices.  At the current stage, we do not have anything significant to say about these posets, so we do not treat them here.  Instead we focus on the combinatorial, topological and enumerative properties of $\Bub(m,n)$.  Computational evidence leads us to believe that the desired lattice $\Lattice_{m,n}$ whose core label order is $\ShufPoset(m,n)$ may not exist.  We finish this article by describing the Galois graph of the bubble lattice in Section~\ref{sec:bubble_galois}.

\section{Preliminaries}

\subsection{Posets}

Throughout this section, we fix a finite set $P$, and consider an order relation on $P$, \ie a reflexive, antisymmetric and transitive binary relation $R\subseteq P\times P$.  We normally write $\leq$ instead of $R$ and call the pair $\Poset=(P,\leq)$ a \defn{poset} (short for \textbf{p}artially \textbf{o}rdered \textbf{set}).  The \defn{dual} poset of $\Poset$ is $\Poset^{d}\defs(P,\geq)$, where the order relations are reversed.  A \defn{minimal} element of $\Poset$ is $m\in P$ such that for all $p\in P$ it holds that $p\leq m$ implies $p=m$.  A \defn{maximal} element of $\Poset$ is a minimal element of $\Poset^{d}$.

For $p,q\in P$ with $p\leq q$, the set $[p,q]\defs\{r\in P\mid p\leq r\leq q\}$ is an \defn{interval} of $\Poset$.  A \defn{covering pair} of $\Poset$ is an interval of cardinality two, and we write $p\lessdot q$ in this case.  The set $\Covers(\Poset)\defs\bigl\{(p,q)\mid p\lessdot q\}$ denotes the cover relation of $\Poset$.  The cover relation comprises a minimal representation of $\Poset$, because the non-covering order relations can be recovered by taking the reflexive and transitive closure of $\Covers(\Poset)$.  The \defn{(Hasse) diagram} of $\Poset$ is the directed graph $\bigl(\Poset,\Covers(\Poset)\bigr)$.  From this perspective, it makes sense to call the covering pairs of $\Poset$ \defn{edges}.
An \defn{edge labeling} of $\Poset$ is any map $\lambda\colon\Covers(\Poset)\to\Lambda$, where $\Lambda$ is an arbitrary (po)set.

A totally ordered subset $C\subseteq P$ is a \defn{chain}.  More precisely, for any $c_{1},c_{2}\in C$ it holds that $c_{1}\leq c_{2}$ or $c_{2}\leq c_{1}$.  The \defn{length} of a chain $C$ is $\lvert C\rvert-1$.  A chain is \defn{maximal} if there does not exist a chain $C'$ of $\Poset$ such that $C\subsetneq C'$.  In other words, a maximal chain can be written as a sequence of covering pairs and contains a minimal and a maximal element of $\Poset$.

If all maximal chains of  $\Poset$ have the same size, then $\Poset$ is \defn{graded}.  In that case, we may define for every $p\in P$ its \defn{rank}: this is simply the maximum length of chain containing $p$ and any minimal element below $p$.  

The converse notion of a chain is an \defn{antichain}, \ie a subset $A\subseteq P$ where \emph{no} distinct members of $A$ are mutually comparable.  Antichains are essentially equivalent to \defn{order ideals}, which are subsets that are downward closed.  More precisely, a subset $I\subseteq P$ is an \defn{order ideal} of $\Poset$ if for all $p\in P$ and all $a\in I$ it holds that $p\leq a$ implies $p\in I$.  The maximal elements of an order ideal form an antichain, and conversely, every antichain $A$ generates the following order ideal
\begin{displaymath}
	P_{\leq A}\defs\bigl\{p\in P\mid p\leq a\;\text{for some}\;a\in A\bigr\}.
\end{displaymath}

\subsection{Lattices}
	\label{sec:lattices}
An \defn{upper bound} for $p,q\in P$ is any element $r\in P$ satisfying $p\leq r$ and $q\leq r$.  If the set of upper bounds has a unique minimal element, then this element is the \defn{join} of $p$ and $q$, and we write $p\vee q$ for this element.  A lower bound for $p$ and $q$ in $\Poset$ is an upper bound for $p$ and $q$ in $\Poset^{d}$, and the meet $p\wedge q$ of $p$ and $q$ (if it exists) in $\Poset$ is the join of $p$ and $q$ in $\Poset^{d}$.  If any two distinct elements of $P$ have a join and a meet, then $\Poset$ is a \defn{lattice}.  

Equivalently, lattices can be viewed as algebraic structures (determined by two binary operations $\vee$ and $\wedge$ satisfying certain axioms) and we are interested in certain classes of lattices.  Nonetheless, we view lattices as poset, and for the remainder of this subsection we consider a finite lattice $\Lattice=(L,\leq)$.

A lattice element $p\in L\setminus\{\hat{0}\}$ is \defn{join-irreducible} if whenever $p=r_{1}\vee r_{2}$ for distinct elements $r_{1},r_{2}\in L$, then $p\in\{r_{1},r_{2}\}$.  Dually, $p$ is \defn{meet-irreducible} if it is join-irreducible in $\Lattice^{d}$.  We write $\JI(\Lattice)$ resp. $\MI(\Lattice)$ for the set of join- resp. meet-irreducible elements of $\Lattice$.  Since we consider only finite lattices, we may easily spot join-irreducible elements in the diagram of $\Lattice$.  These are the elements covering a unique element.  In other words, if $j\in \JI(\Lattice)$, then there exists a unique $j_{*}\in L$ such that $(j_{*},j)\in\Covers(\Lattice)$.

Two covering pairs $(p,q),(p',q')\in\Covers(\Lattice)$ are \defn{perspective} if either $q\vee p'=q'$ and $q\wedge p'=p$ or $q'\vee p=q$ and $q'\wedge p=p'$.  See Figure~\ref{fig:perspective_covers} for an illustration.

\begin{figure}
	\centering
	\includegraphics[page=2,scale=1]{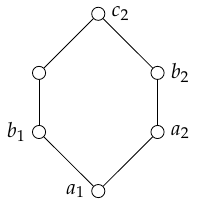}
	\caption{The covering pairs $(a_{1},b_{1})$ and $(b_{2},c_{2})$ are perspective.  The covering pair $(a_{2},b_{2})$ is not perspective to any other covering pair.}
	\label{fig:perspective_covers}
\end{figure}

\subsubsection{Semidistributive Lattices}

The lattice $\Lattice$ is \defn{join semidistributive} if for all $p,q,r\in L$ it holds that 
\begin{equation}\label{eq:jsd}
	p\vee q=p\vee r\quad\text{implies}\quad p\vee q=p\vee(q\wedge r)
\end{equation}
Moreover, $\Lattice$ is \defn{meet semidistributive} if $\Lattice^{d}$ is join semidistributive, and $\Lattice$ is \defn{semidistributive} if it is both join and meet semidistributive.

A join semidistributive lattice admits nice canonical forms for its elements.  For $p\in L$, a set $X\subseteq L$ is a \defn{join representation} of $p$ if $p=\bigvee X$.  Such a join representation is \defn{irredundant} if no proper subset of $X$ joins to $p$.  Moreover, $X$ \defn{join refines} another join representation $X'$ of $p$ if $L_{\leq X}\subseteq L_{\leq X'}$.  The unique minimal element in the set of join representations of $p$ under this join refinement (if it exists) is the \defn{canonical join representation} of $p$; denoted by $\canset(p)$.  We have the following characterization of finite join semidistributive lattices.

\begin{theorem}[\cite{freese:free}*{Theorem~2.24}]\label{thm:jsd_characterization}
	A finite lattice is join semidistributive if and only if every element admits a canonical join representation.
\end{theorem}

The members of $\canset(p)$ are the \defn{canonical joinands} of $p$.  It is quickly verified that canonical joinands are join-irreducible.  There is a nice combinatorial way to compute the canonical joinands of $p$ using a particular edge labeling of $\Lattice$.  We consider
\begin{equation}\label{eq:jsd_labeling}
	\jsdlabeling\colon\Covers(\Lattice)\to\JI(\Lattice), (p,q)\mapsto\bigwedge\{r\in L\mid p\vee r=q\}.
\end{equation}
If $\Lattice$ is join semidistributive, then this labeling is well defined (in the sense that its image is indeed contained in $\JI(\Lattice)$); see\cite{adaricheva:join}*{Lemma~1.8}.  In fact, we get much more.

\begin{proposition}[\cite{barnard:canonical}*{Lemma~19}]\label{prop:join_semidistributive_canonical}
	If $\Lattice=(L,\leq)$ is join-semidistributive, then the canonical join representation of $p\in L$ is $\bigl\{\jsdlabeling(p',p)\mid p'\lessdot p\bigr\}$.
\end{proposition}

The labeling $\jsdlabeling$ has the following property, which is similar to Lemma~\ref{lem:cu_labeling_perspective} below.

\begin{lemma}[\cite{muehle:distributive}*{Lemma~3.3}]\label{lem:jsd_labeling_perspective}
	Let $\Lattice=(L,\leq)$ be a finite join-semidistributive lattice and let $(p,q)\in\Covers(\Lattice)$.  Then $\jsdlabeling(p,q)=j$ if and only if $(p,q)$ and $(j_{*},j)$ are perspective.
\end{lemma}
%

\begin{lemma}[\cite{freese:free}*{Corollary~2.55}]
	If $\Lattice$ is semidistributive, then $\bigl\lvert\JI(\Lattice)\bigr\rvert=\bigl\lvert\MI(\Lattice)\bigr\rvert$.
\end{lemma}

\subsubsection{Lattices Constructable by Interval Doublings}

Given two posets $\Poset_{1}=(P_{1},\leq_{1})$ and $\Poset_{2}=(P_{2},\leq_{2})$, the \defn{direct product} is the poset $\Poset\defs(P_{1}\times P_{2},\leq)$, where $(p_{1},p_{2})\leq(q_{1},q_{2})$ if and only if $p_{1}\leq_{1}q_{1}$ and $p_{2}\leq_{2}q_{2}$.  It is easy to show that the direct product of two lattices is again a lattice.  Given an integer $n>0$, we write $[n]\defs\{1,2,\ldots,n\}$.  We write $\mathbf{n}=\bigl([n],\leq)$, where $\leq$ is the usual order on the integers.  

For a subset $X\subseteq P$, the \defn{doubling} of $\Poset$ by $X$ is the induced subposet of the direct product $\Poset\times\mathbf{2}$ determined by the ground set
\begin{displaymath}
	\Bigl(P_{\leq X}\times\{1\}\Bigr)\uplus\Bigl(\bigl((P\setminus P_{\leq X})\cup X\bigr)\times\{2\}\Bigr).
\end{displaymath}
We write $\Poset[X]$ for the resulting poset.  See Figure~\ref{fig:doubling} for an illustration.

\begin{figure}
    \centering
	\includegraphics[page=3,scale=1]{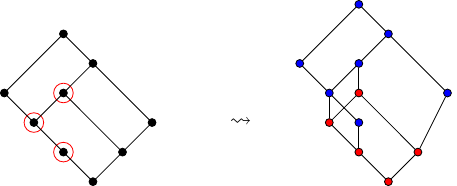}
    \caption{Doubling the poset on the left by the highlighted interval yields the poset on the right.}
    \label{fig:doubling}
\end{figure}

If $\Poset$ is a lattice and $X$ is order convex, then the doubling $\Poset[X]$ is again a lattice~\cite{day:characterizations}.  Of particular interest are the lattices which are \defn{constructable by interval doublings}, \ie lattices that arise from the singleton lattice by a sequence of interval doublings.  Such lattices are characterized by the existence of a certain edge labeling.

Let $\Lattice=(L,\leq)$ be a finite lattice.  An interval $I=[p,q]$ in $\Lattice$ is \defn{polygonal} if it is the union of two chains whose only common elements are $p$ and $q$.  In particular, there exist exactly two elements $p_{1},p_{2}\in L$ such that $p\lessdot p_{i}\leq q$ for $i\in\{1,2\}$ and $p_{1}\vee p_{2}=q$.  We say that $I$ is \defn{generated} by $p,p_{1},p_{2}$.
Then, an edge labeling $\lambda$ of $\Lattice$ is a \defn{CU-labeling} if the following properties are satisfied for every polygonal interval (which is supposed to be generated by $p,p_{1},p_{2}$):
\begin{description}
	\item[CU1\label{it:cu1}] the elements $q_{1}\in C_{1}$ and $q_{2}\in C_{2}$ such that $q_{1}\lessdot p_{1}\vee p_{2}$ and $q_{2}\lessdot p_{1}\vee p_{2}$ satisfy $\lambda(p,p_{1})=\lambda(q_{2},p_{1}\vee p_{2})$ and $\lambda(p,p_{2})=\lambda(q_{1},p_{1}\vee p_{2})$;
	\item[CU2\label{it:cu2}] if $r,r'\in C_{1}$ (resp. $r,r'\in C_{2}$) with $p<r\lessdot r'<p_{1}\vee p_{2}$, then $\lambda(p,p_{1})<\lambda(r,r')$ and $\lambda(p,p_{2})<\lambda(r,r')$;
	\item[CU3\label{it:cu3}] the labels appearing on $C_{1}$ (resp. $C_{2}$) are pairwise distinct;
	\item[CU4\label{it:cu4}] $\lambda(j_{*},j)\neq\lambda(j'_{*},j')$ for all $j,j'\in\JI(\Poset)$ with $j\neq j'$;
	\item[CU5\label{it:cu5}] $\lambda(m,m^{*})\neq\lambda(m',(m')^{*})$ for all $m,m'\in\MI(\Poset)$ with $m\neq m'$.
\end{description}

We have the following connection.

\begin{theorem}[\cite{garver.mcconville:oriented_tree}*{Proposition~2.5}]\label{thm:congruence_uniform_labeling}
	A finite lattice is constructable by interval doublings if and only if it admits a CU-labeling.
\end{theorem}

\begin{remark}
	Lattices that are constructable by interval doublings have a huge significance in lattice theory.  By \cite{day:characterizations}, these are precisely the lattices which are bounded-homomorphic images of free lattices.  Another interesting property of such lattices is the fact that there is a bijection between join-irreducible elements and join-irreducible congruence relations.  For more details on these properties, we refer the interested reader to \cites{day:doubling,day:congruence,day:characterizations,freese:free}.
\end{remark}

\begin{remark}
	An edge labeling satisfying conditions \eqref{it:cu1}--\eqref{it:cu3} characterizes those lattices that are constructable by doublings by convex sets, see \cite{reading:lattice}*{Theorem~4}.
\end{remark}

Let us record two other properties of interval-constructable lattices.

\begin{theorem}[\cite{day:characterizations}*{Lemma~4.2}]\label{thm:congruence_uniform_semidistributive}
	If a finite lattice is constructable by interval doublings, then it is semidistributive.
\end{theorem}


\begin{lemma}[\cite{garver.mcconville:oriented_tree}*{Lemma~2.6}]\label{lem:cu_labeling_perspective}
	Let $\Lattice=(L,\leq)$ be an interval-constructable lattice with a surjective CU-labeling $\lambda\colon\Covers(\Lattice)\to\Poset$ for some poset $\Poset=(P,\preceq)$.  For every $s\in P$ there exists a unique $j\in\JI(\Lattice)$ and a unique $m\in\MI(\Lattice)$ such that $\lambda(j_{*},j)=s=\lambda(m,m^{*})$.  Moreover, for every $(p,q)\in\Covers(\Lattice)$ it holds that $\lambda(p,q)=s$ if and only if $(p,q)$ is perspective with $(j_{*},j)$ and $(m,m^{*})$.
\end{lemma}

\begin{corollary}\label{cor:cu_jsd_labeling}
    Any CU-labeling of a interval-constructable lattice is combinatorially equivalent to $\jsdlabeling$.
\end{corollary}
\begin{proof}
    By Theorem~\ref{thm:congruence_uniform_semidistributive}, every interval-constructable lattice is semidistributive so that $\jsdlabeling$ is well defined.  Now, Lemmas~\ref{lem:jsd_labeling_perspective} and \ref{lem:cu_labeling_perspective} imply that $\jsdlabeling$ and any CU-labeling are determined by the perspectivity relation.
\end{proof}

\subsubsection{Trim Lattices}

If $\Lattice$ has length $k$, then it is straightforward to verify that 
\begin{equation}\label{eq:lattice_length_irreds}
	k\leq\min\bigl\{\bigl\lvert\JI(\Lattice)\bigr\rvert,\bigl\lvert\MI(\Lattice)\bigr\rvert\bigr\}.
\end{equation}
In the case where $\bigl\lvert\JI(\Lattice)\bigr\rvert=k=\bigl\lvert\MI(\Lattice)\bigr\rvert$, $\Lattice$ is called \defn{extremal}~\cite{markowsky:primes}.  By \cite{markowsky:primes}*{Theorem~14(ii)}, every finite lattice can be embedded as an interval in a finite extremal lattice, which implies that extremality is not inherited by intervals.  In \cite{thomas:analogue}, a strengthening of extremality was introduced which does possess this hereditary property.

An element $p\in L$ is \defn{left modular} if for any $r<q$ it holds that
\begin{displaymath}
	(r\vee p)\wedge q = r\vee(p\wedge q).
\end{displaymath}
If $\Lattice$ has length $k$, and possesses a maximal chain of length $k$ comprised entirely of left-modular elements, then $\Lattice$ is itself \defn{left modular}.  A lattice that is both extremal and left modular is \defn{trim}.  

While there is no general relation between the classes of semidistributive and trim lattices, there is a remarkable connection.  

\begin{theorem}[\cite{thomas:rowmotion}*{Theorem~1.4}]\label{thm:semidistributive_extremal_trim}
	Every semidistributive, extremal lattice is trim.
\end{theorem}

This result is extremely useful because it is often much simpler to establish extremality than it is to establish left-modularity.  Recently, the converse of Theorem~\ref{thm:semidistributive_extremal_trim} was proven in \cite{muehle:extremality}.

\subsection{Shuffle Words}

For nonnegative integers $m$ and $n$, we consider two disjoint sets of letters: 
\begin{displaymath}
	X = \{x_{1},x_{2},\ldots,x_{m}\}\quad\text{and}\quad Y=\{y_{1},y_{2},\ldots,y_{n}\}.
\end{displaymath}
A word over $X\cup Y$ is \defn{simple} if it does not contain duplicate letters. The \defn{support} of a word is the set of letters it contains. We may apply set-theoretic operators to a word $\wbf$ without explicitly referring to its support. For example, for any word $\wbf$ and any letter $w\in X\cup Y$, we write ``$w\in\wbf$'' as an abbreviation for ``$w$ is in the support of $\wbf$''. The empty word is denoted by $\epsilon$ and we usually (at least in the examples) write the letters of $\xbf$ in red and the letters of $\ybf$ in blue. 

A \defn{subword} of a simple word $\wbf=w_{1}w_{2}\cdots w_{k}$ is any word of the form $w_{i_{1}}w_{i_{2}}\cdots w_{i_{l}}$ with $i_{1}<i_{2}<\cdots<i_{l}$.  For $i\in[k]\defs\{1,2,\ldots,k\}$ we write $\wbf_{\hat{\imath}}$ for the subword of $\wbf$ obtained by deleting the letter $w_{i}$.  

If $\ubf,\vbf$ are simple words, then the \defn{restriction} of $\ubf$ to $\vbf$, denoted by $\ubf_{\vbf}$, is the subword of $\ubf$ formed by the common letters of $\ubf$ and $\vbf$.  For instance, if $\ubf=\rd{x_1}\bl{y_1}\rd{x_2x_3}\bl{y_3}$ and $\vbf=\rd{x_3}\bl{y_1}\rd{x_4}$, then the restriction of $\ubf$ by $\vbf$ is $\ubf_{\vbf} = \bl{y_1}\rd{x_3}$.

Our main interest lies in \emph{order-preserving} simple words.  That means, if we define 
\begin{displaymath}
	\xbf\defs x_{1}x_{2}\cdots x_{m}\quad\text{and}\quad\ybf\defs y_{1}y_{2}\cdots y_{n},
\end{displaymath}
then we consider the set of simple words $\wbf$ with the property that $\wbf_{\xbf}$ is a subword of $\xbf$ and $\wbf_{\ybf}$ is a subword of $\ybf$.  We call such words \defn{shuffle words} of $\xbf$ and $\ybf$, and we write $\Shuf(m,n)$ for the set of all shuffle words.  It is easy to see that the number of shuffle words depends only on the cardinalities of $X$ and $Y$, and not on the concrete elements of $X$ and $Y$.

Let $\ubf=u_{1}u_{2}\cdots u_{k}\in\Shuf(m,n)$.  Lemma~4.6 in \cite{greene:shuffle} states that $\ubf$ is uniquely determined by its \defn{interface}, \ie the set of letters $x\in X$, $y\in Y$ for which there exists $i\in[k-1]$ such that $u_{i}=y$ and $u_{i+1}=x$, and its \defn{residue}, \ie the letters of $\ubf$ which are not in the interface.  This motivates the following two operations on $\Shuf(m,n)$.  

An \defn{indel} is a relation $\ubf\indel\ubf_{\hat{\imath}}$ if $u_{i}\in X$ or $\ubf_{\hat{\imath}}\indel\ubf$ if $u_{i}\in Y$.  In other words, an indel of $\ubf$ is a shuffle word obtained from $\ubf$ by either \textbf{in}serting an element of $Y$ or \textbf{del}eting an element of $X$.  A \defn{(forward) transposition} swaps two letters $u_{i}$ and $u_{i+1}$ if $u_{i}\in X$ and $u_{i+1}\in Y$.
In this situation, we write $\ubf\transpose\ubf'$, where $\ubf'=u_{1}u_{2}\cdots u_{i+1}u_{i}\cdots u_{k}$.  

Drawing inspiration from the analogous situation for permutations, we define the \defn{inversion set} of $\ubf$ by
\begin{displaymath}
	\invset(\ubf) \defs \bigl\{(x_{s},y_{t})\mid\;\text{there exist}\;i<j\;\text{such that}\;u_{i}=y_{t}\;\text{and}\;u_{j}=x_{s}\bigr\}.
\end{displaymath}
The inversion set provides a little bit more information than the interface, because it exactly locates the inversions.  However, knowing interface and residue enables us to compute the inversion set.  Table~\ref{tab:inversions_21} lists the elements of $\Shuf(2,1)$ together with their inversion sets.

\begin{table}
    \centering
    \begin{tabular}{r|l}
        $\ubf\in\Shuf(2,1)$ & $\invset(\ubf)$\\
        \hline\hline
        $\epsilon$ & $\emptyset$\\
        $\rd{x_1}$ & $\emptyset$\\
        $\rd{x_2}$ & $\emptyset$\\
        $\bl{y_1}$ & $\emptyset$\\
        $\rd{x_1x_2}$ & $\emptyset$\\
        $\rd{x_1}\bl{y_1}$ & $\emptyset$\\
        $\rd{x_2}\bl{y_1}$ & $\emptyset$\\
        $\bl{y_1}\rd{x_1}$ & $\bigl\{(\rd{x_1},\bl{y_1})\bigr\}$\\
        $\bl{y_1}\rd{x_2}$ & $\bigl\{(\rd{x_2},\bl{y_1})\bigr\}$\\
        $\rd{x_1x_2}\bl{y_1}$ & $\emptyset$\\
        $\rd{x_1}\bl{y_1}\rd{x_2}$ & $\bigl\{(\rd{x_2},\bl{y_1})\bigr\}$\\
        $\bl{y_1}\rd{x_1x_2}$ & $\bigl\{(\rd{x_1},\bl{y_1}),(\rd{x_2},\bl{y_1})\bigr\}$\\
    \end{tabular}
    \caption{The elements of $\Shuf(2,1)$ together with their inversion sets.}
    \label{tab:inversions_21}
\end{table}

Now finally, we use indels and transpositions to define two partial orders on $\Shuf(m,n)$.  The \defn{shuffle order}, denoted by $\shufleq$, is the reflexive and transitive closure of indels, and the \defn{bubble order}, denoted by $\bubleq$, is the reflexive and transitive closure of indels and transpositions\footnote{The name ``bubble order'' is to emphasize that this is an order extension of the shuffle order in which we have to bubble sort the words before we can perform an indel.}.  We write $\ShufPoset(m,n)\defs\bigl(\Shuf(m,n),\shufleq\bigr)$ and $\Bub(m,n)\defs\bigl(\Shuf(m,n),\bubleq\bigr)$ for the corresponding partially ordered sets (\defn{posets}).  Figure~\ref{fig:shuffle_21} shows $\ShufPoset(2,1)$ and Figure~\ref{fig:bubble_21} shows $\Bub(2,1)$.  The poset $\ShufPoset(m,n)$ was intensively studied in \cite{greene:shuffle}, while the poset $\Bub(m,n)$ is new.  The main purpose of this article is to exhibit several remarkable structural and enumerative correspondences between these two posets.

\section{The Order and Cover Relation of $\Bub(m,n)$}

In this section we characterize order relation of $\Bub(m,n)$ and deduce a characterization of its covering pairs.  We then derive some first enumerative properties. 
%
%
The shuffle order was defined in \cite{greene:shuffle} by the following conditions: two shuffle words $\ubf,\vbf\in\Shuf(m,n)$ satisfy $\ubf\shufleq\vbf$ if and only if 
\begin{itemize}
	\item $\vbf_{\xbf}$ is a subword of $\ubf_{\xbf}$,
	\item $\ubf_{\ybf}$ is a subword of $\vbf_{\ybf}$,
	\item $\ubf_{\vbf}=\vbf_{\ubf}$.
\end{itemize}
We now give an analogous characterization of the bubble order.

\begin{lemma}\label{lem:bubble_order}
	Let $\ubf,\vbf\in\Shuf(m,n)$.  Then, $\ubf\bubleq\vbf$ if and only if:
	\begin{itemize}
		\item $\vbf_{\xbf}$ is a subword of $\ubf_{\xbf}$,
		\item $\ubf_{\ybf}$ is a subword of $\vbf_{\ybf}$, and 
		\item $\invset(\ubf_{\vbf})\subseteq\invset(\vbf_{\ubf})$.
	\end{itemize}
\end{lemma}
\begin{proof}
	Let $\ubf,\vbf\in\Shuf(m,n)$ be such that $\ubf\bubleq\vbf$. This means that there exists a sequence $\ubf=\ubf^{(0)},\ldots,\ubf^{(k)}=\vbf$ such that for all $i$, either $\ubf^{(i-1)}\indel\ubf^{(i)}$ or $\ubf^{(i-1)}\transpose\ubf^{(i)}$. If $k=0$, then the desired conditions all hold. We assume $k>0$.

	Since transposition does not change the restrictions to $\xbf$ or $\ybf$, and indels either delete letters from $\xbf$ or insert letters from $\ybf$, it follows by induction that $\vbf_{\xbf}$ is a subword of $\ubf_{\xbf}$ and $\ubf_{\ybf}$ is a subword of $\vbf_{\ybf}$.

	To check the last condition, set $\wbf=\ubf^{(k-1)}$ and assume that $\invset(\ubf_{\wbf})$ is a subset of $\invset(\wbf_{\ubf})$. If $\wbf\transpose\vbf$, then either $\invset(\wbf_{\ubf})=\invset(\vbf_{\ubf})$, or one new inversion is added. Also, $\wbf\transpose\vbf$ implies $\invset(\ubf_{\wbf})=\invset(\ubf_{\vbf})$ since $\wbf$ and $\vbf$ are composed of the same set of letters. Hence, $\invset(\ubf_{\vbf})$ is a subset of $\invset(\vbf_{\ubf})$.

	Finally, suppose $\wbf\indel\vbf$. If $\vbf$ is obtained from $\wbf$ by inserting $y_t$, then $\invset(\wbf_{\ubf})=\invset(\vbf_{\ubf})$ and $\invset(\ubf_{\wbf})=\invset(\ubf_{\vbf})$ since this new letter is not in $\ubf$. If $\vbf$ is obtained from $\wbf$ by deleting $x_s$, then both $\invset(\ubf_{\vbf})$ and $\invset(\vbf_{\ubf})$ can be computed by deleting all inversions of $\ubf$ and $\wbf$ that involve the letter $x_s$ from $\invset(\ubf_{\wbf})$ and $\invset(\wbf_{\ubf})$, respectively.

	\medskip
	
	Now assume that $\ubf$ and $\vbf$ are words such that the following three properties hold:
	\begin{itemize}
		\item $\vbf_{\xbf}$ is a subword of $\ubf_{\xbf}$,
		\item $\ubf_{\ybf}$ is a subword of $\vbf_{\ybf}$, and
		\item $\invset(\ubf_{\vbf})$ is a subset of $\invset(\vbf_{\ubf})$.
	\end{itemize}
	Let $\ubf_{\xbf}\setminus\vbf_{\xbf}=\{x_{i_{1}},x_{i_{2}},\ldots,x_{i_{s}}\}$.  We set $\ubf^{(0)}=\ubf$ and construct $\ubf^{(r)}$ from $\ubf^{(r-1)}$ by deleting the letter $x_{i_{r}}$.  Then, we have $\ubf\indel\ubf^{(1)}\indel\cdots\indel\ubf^{(r)}$.  If we set $\wbf=\ubf^{(r)}$, then we have $\wbf_{\xbf}=\vbf_{\xbf}$ and $\ubf\bubleq\wbf$.  The inversion set of $\wbf$ can be computed from the inversion set of $\ubf$ by deleting all inversions involving any of the letters $x_{i_{1}},x_{i_{2}},\ldots,x_{i_{s}}$.  In particular, if $(x,y)\in\invset(\wbf_{\vbf})$, then $x$ and $y$ must be letters of $\wbf$ and $\vbf$, which means that $(x,y)\in\invset(\ubf_{\vbf})\subseteq\invset(\vbf_{\ubf})$ by assumption.  But since $x$ and $y$ are letters of $\wbf$, it follows that $(x,y)\in\invset(\vbf_{\wbf})$.  We therefore get $\invset(\wbf_{\vbf})\subseteq\invset(\vbf_{\wbf})$.  
	
	Now, say that $\vbf_{\ybf}\setminus\ubf_{\ybf}=\{y_{j_{1}},y_{j_{2}},\ldots,y_{j_{t}}\}$.  We set $\vbf^{(t)}=\vbf$ and we obtain $\vbf^{(r-1)}$ from $\vbf^{(r)}$ by deleting the letter $y_{j_{r}}$.  Thus, we get $\vbf^{(0)}\indel\vbf^{(1)}\indel\cdots\indel\vbf^{(t)}$.  If we set $\tilde{\wbf}=\vbf^{(0)}$, then we have $\tilde{\wbf}_{\ybf}=\ubf_{\ybf}$ and $\tilde{\wbf}\bubleq\vbf$.  In fact, $\wbf$ and $\tilde{\wbf}$ have the same support.  Moreover, the inversion set of $\tilde{\wbf}$ can be computed from $\invset(\vbf)$ by deleting all inversions involving any of the letters $y_{j_{1}},y_{j_{2}},\ldots,y_{j_{t}}$.  So if $(x,y)\in\invset(\tilde{\wbf}_{\vbf})$, then necessarily $(x,y)\in\invset(\vbf_{\tilde{\wbf}})$, because $\vbf$ is obtained from $\tilde{\wbf}$ by inserting letters of $\ybf$, without changing any of the relative positions of the other letters.  Moreover, if $(x,y)\in\invset(\wbf)$, then $(x,y)\in\invset(\ubf_{\vbf})$, because $x$ and $y$ must be letters of both $\ubf$ and $\vbf$ and $\wbf$ is obtained from $\ubf$ by deleting letters of $\xbf$.  By assumption, we have $(x,y)\in\invset(\vbf_{\ubf})$, which implies $(x,y)\in\invset(\tilde{\wbf})$.  It follows that $\invset(\wbf)\subseteq\invset(\tilde{\wbf})$.  Since $\wbf$ and $\tilde{\wbf}$ have the same support and $\invset(\wbf)\subseteq\invset(\tilde{\wbf})$, we can obtain $\tilde{\wbf}$ from $\wbf$ by a sequence of transpositions, which yields $\wbf\bubleq\tilde{\wbf}$.  In summary, w get $\ubf\bubleq\wbf\bubleq\tilde{\wbf}\bubleq\vbf$ as desired.
\end{proof}

\begin{remark}
	If $m=0$ (resp. $n=0$), then the corresponding bubble poset is isomorphic to the Boolean lattice with $2^{n}$ (resp. $2^{m}$) elements.
\end{remark}

\begin{corollary}\label{cor:bubble_bounded}
	For $m,n\geq 0$, the poset $\Bub(m,n)$ is bounded with bottom element $\xbf$ and top element $\ybf$.
\end{corollary}

\begin{lemma}\label{lem:bubble_duality}
	For $m,n\geq 0$, the poset $\Bub(m,n)$ is dual to $\Bub(n,m)$.
\end{lemma}
\begin{proof}
	The map that exchanges $x$'s for $y$'s is clearly a bijection from $\Shuf(m,n)$ to $\Shuf(n,m)$.  Let $\ubf,\vbf\in\Shuf(m,n)$ and let $\bar{\ubf},\bar{\vbf}$ be the corresponding elements from $\Shuf(n,m)$.  Suppose that $\ubf\bubleq\vbf$.  We use Lemma~\ref{lem:bubble_order} to see that $\vbf_{\xbf}$ is a subword of $\ubf_{\xbf}$ and $\ubf_{\ybf}$ is a subword of $\vbf_{\ybf}$.  But then, $\bar{\vbf}_{\ybf}$ is a subword of $\bar{\ubf}_{\ybf}$ and $\bar{\ubf}_{\xbf}$ is a subword of $\bar{\vbf}_{\xbf}$.  Moreover, if $(x_{s},y_{t})\in\invset(\ubf)$, then $x_{s}$ appears after $y_{t}$ in $\ubf$.  This means that $y_{s}$ appears after $x_{t}$ in $\bar{\ubf}$ and we get that
	\begin{displaymath}
		\invset(\bar{\ubf}) = \bigl\{(x_{s},y_{t})\mid (x_{t},y_{s})\notin\invset(\ubf)\bigr\}.
	\end{displaymath}
	Thus, $\invset(\ubf_{\vbf})\subseteq\invset(\vbf_{\ubf})$ implies that $\invset(\bar{\ubf}_{\bar{\vbf}})\supseteq\invset(\bar{\vbf}_{\bar{\ubf}})$.  By Lemma~\ref{lem:bubble_order}, $\bar{\vbf}\bubleq\bar{\ubf}$ in $\Bub(n,m)$.  Applying the same argument in reverse finishes the proof.
\end{proof}

We now want to describe the covering pairs in $\Bub(m,n)$, \ie we want to understand the relations $\ubf\bubless\vbf$ such that there exists no $\wbf$ with $\ubf\bubless\wbf\bubless\vbf$.  We write $\ubf\bubcov\vbf$ in that event.

While it is easily checked that every transposition corresponds to a covering pair in $\Bub(m,n)$, the same is not true for arbitrary indels. In fact, let $\ubf=\rd{x_1x_2}\bl{y_1}$ and $\vbf=\rd{x_1}\bl{y_1}$.  Then, $\ubf\indel\vbf$, because we delete the letter $\rd{x_2}$.  However, this is not a covering pair, because the word $\wbf=\rd{x_1}\bl{y_1}\rd{x_2}$ lies strictly between $\ubf$ and $\vbf$, see Figure~\ref{fig:bubble_21}.  In fact, we have $\ubf\transpose\wbf\indel\vbf$.  In some sense, transpositions are prioritized over deletions.

Let $\ubf=u_{1}u_{2}\cdots u_{k}$. Recall that for $i\in[k]$ we denote by $\ubf_{\hat{\imath}}$ the word obtained by deleting the letter $u_{i}$.  Then, $\ubf_{\hat{\imath}}\indel\ubf$ if $u_{i}\in Y$ and $\ubf\indel\ubf_{\hat{\imath}}$ if $u_{i}\in X$.  We define a new relation $\sindel$ on $\Shuf(m,n)$ by setting
\begin{displaymath}
	\vbf\sindel\vbf'\quad\text{if and only if}\quad
		\begin{cases}
			\vbf=\ubf\;\text{and}\;\vbf'=\ubf_{\hat{\imath}}, & \text{if}\;u_{i},u_{i+1}\in X,\\
			\vbf=\ubf_{\hat{\imath}}\;\text{and}\;\vbf'=\ubf, & \text{if}\;u_{i},u_{i+1}\in Y.
	    \end{cases}
\end{displaymath}
In both cases, if $i=k$, then we just have to check the condition for $u_{k}$.  We call $\sindel$ a \defn{right indel}, because the inserted (resp. deleted) letter is pushed as far right as possible.  Clearly, $\ubf\sindel\vbf$ implies $\ubf\indel\vbf$, but the converse is not true.

\begin{lemma}\label{lem:bubble_covers}
	For $\ubf,\vbf\in\Shuf(m,n)$ we have $\ubf\bubcov\vbf$ if and only if either $\ubf\transpose\vbf$ or $\ubf\sindel\vbf$.
\end{lemma}
\begin{proof}
    Suppose that $\ubf\bubcov\vbf$.  By construction, the supports of $\ubf$ and $\vbf$ differ by at most one element.  If $\ubf$ and $\vbf$ have the same support, \ie $\ubf_{\xbf}=\vbf_{\xbf}$ and $\ubf_{\ybf}=\vbf_{\ybf}$.  But then, Lemma~\ref{lem:bubble_order} implies that $\invset(\ubf)=\invset(\vbf)\setminus\{(x_{s},y_{t})\}$ which yields $\ubf\transpose\vbf$.  Now suppose that there exists $x_{s}\in\ubf$ such that $x_{s}\notin\vbf$.  It follows that $\ubf_{\ybf}=\vbf_{\ybf}$.  Now, if $x_{s}$ is followed by some $y_{t}$, then we may consider the word $\wbf$ obtained from $\ubf$ by transposing $x_{s}$ and $y_{t}$.  We have $\wbf_{\xbf}=\ubf_{\xbf}\subsetneq\vbf_{\xbf}$ and $\wbf_{\ybf}=\ubf_{\ybf}=\vbf_{\ybf}$ and $\invset(\wbf_{\vbf})=\invset(\ubf_{\vbf})\subseteq\invset(\vbf_{\ubf})=\invset(\ubf_{\wbf})$.  It follows that $\ubf\bubless\wbf\bubless\vbf$, a contradiction.  Thus, $x_{s}$ must be followed by some $x_{s'}$ and we get $\ubf\sindel\vbf$.

    For the converse, first suppose that $\ubf\transpose\vbf$.  Then $\ubf$ and $\vbf$ have the same set of letters and there is a unique inversion in $\vbf$ that is not an inversion of $\ubf$.  Then, it is immediate that $\ubf\bubcov\vbf$, because we can neither apply another transposition to $\ubf$ or insert/delete a letter in $\ubf$ so that we stay below $\vbf$.
	
	Lastly, suppose that $\ubf\sindel\vbf$, where $\vbf$ is obtained from $\ubf$ by deleting the letter $x_{s}$.  By construction, $x_{s}$ is followed by $x_{s'}$ in $\ubf$.  Assume that there is some $\wbf\in\Shuf(m,n)$ with $\ubf\bubcov\wbf\bubless\vbf$.  In view of the first part of this proof, we get either $\ubf\transpose\wbf$ or $\ubf\sindel\wbf$.  However, since $\wbf_{\xbf}\subseteq\vbf_{\xbf}=\ubf_{\xbf}\setminus\{x_{s}\}$ the second case cannot happen.
	Thus, $\ubf\transpose\wbf$, and suppose that $\ubf=u_{1}u_{2}\cdots u_{k}$, where $u_{i}=x_{s}$.  We may therefore write $\invset(\wbf)=\invset(\ubf)\cup\{(u_{j},u_{j+1})\}$, for some $j$.  However, since $\ubf\sindel\vbf$, we need to have $u_{i+1}=x_{s'}$ so that $u_{j},u_{j+1}\neq x_{s}$.  Now, $\wbf\bubless\vbf$ forces $(u_{j},u_{j+1})\in\invset(\vbf)\subseteq\invset(\ubf)$, a contradiction.  It follows that $\wbf$ cannot exist, and we conclude $\ubf\bubcov\vbf$.  The reasoning, when $\vbf$ is obtained from $\ubf$ by inserting a letter $y_{t}$ is analogous.
\end{proof}

A poset is \defn{Hasse-regular} of degree $k$ if its diagram, viewed as a simple graph, is $k$-regular.  Figure~\ref{fig:bubble_21} shows that $\Bub(2,1)$ is Hasse-regular of degree $3$.  This is not a coincidence.

\begin{lemma}\label{lem:bubble_hasse_regular}
	The poset $\Bub(m,n)$ is Hasse-regular of degree $m+n$.
\end{lemma}
\begin{proof}
	Let $\ubf=u_{1}u_{2}\cdots u_{k}$.  Let $x\in\xbf$ be such that $x=u_{i}$ for some $i$.  If $i=k$ or $u_{i+1}$ is a letter of $\xbf$, then we obtain $\ubf'$ by performing a right indel deleting $x$.  If $i<k$ and $u_{i+1}$ is a letter of $\ybf$, then we obtain $\ubf'$ by transposing $u_{i}$ and $u_{i+1}$.  By Lemma~\ref{lem:bubble_covers}, we get $\ubf\bubcov\ubf'$.  Thus, every letter of $\xbf$ contained in $\ubf$ generates a unique element $\ubf'$ covering $\ubf$.  
	
	Now, let $y_{t}\in\ybf$ be such that $y\neq u_{i}$ for all $i$.  If for all $t'>t$, $y_{t'}\notin\ubf$, then we obtain $\ubf'$ by performing a right indel inserting $y_{t}$ at the end of $\ubf$.  Otherwise, there exists a smallest $t'$ with $t'>t$ such that $y_{t'}$ is a letter of $\ubf$.  We obtain $\ubf'$ by performing a right indel inserting $y_{t}$ right before $y_{t'}$.  By Lemma~\ref{lem:bubble_covers}, we get $\ubf\bubcov\ubf'$.  In other words, every letter of $\ybf$ \emph{not} contained in $\ubf$ generates a unique element $\ubf'$ covering $\ubf$.  
	
	We notice that none of the elements obtained from $\ubf$ by deleting or transposing at some letter of $\xbf$ agrees with an element obtained from $\ubf$ by inserting some letter of $\ybf$.  Thus, if $\ubf_{\xbf}$ has $a$ elements and $\ubf_{\ybf}$ has $b$ elements, then we get $a+(n-b)$ distinct elements covering $\ubf$.  By Lemma~\ref{lem:bubble_duality}, we obtain $(m-a)+b$ distinct elements that are covered by $\ubf$.  Therefore, the element $\ubf$ has degree $a+n-b+m-a+b=m+n$.
\end{proof}

We summarize the following part of the proof of Lemma~\ref{lem:bubble_hasse_regular} for later use.

\begin{corollary}\label{cor:xy_removal}
	Let $\vbf$ in $\Shuf(m,n)$, and let $x$ (resp. $y$) be an arbitrary letter of $\xbf$ (resp. $\ybf$).  
	\begin{itemize}
		\item If $x\notin\vbf$, then there exists a unique word $\ubf$ with $\ubf\bubcov\vbf$ such that $x\in\ubf$.
		\item If $y\in\vbf$, then there exists a unique word $\ubf$ with $\ubf\bubcov\vbf$ such that either $y\notin\ubf$ or $\ubf_{\xbf}=\vbf_{\xbf}$, $\ubf_{\ybf}=\vbf_{\ybf}$ and $\invset(\ubf)=\invset(\vbf)\setminus\bigl\{(x',y)\bigr\}$ for $x'$ immediately succeeding $y$ in $\vbf$.
	\end{itemize}
%
\end{corollary}

This description of the covering pairs in $\Bub(m,n)$ has the following consequence for the subposet of shuffle words with fixed support.  For $s\in[m]$ and $t\in[n]$, we consider (arbitrary) subwords $\xbf^{(s)}=x_{i_{1}}x_{i_{2}}\cdots x_{i_{s}}$ of $\xbf$ and $\ybf^{(t)}=y_{j_{1}}y_{j_{2}}\cdots y_{j_{t}}$ of $\ybf$.  Let
\begin{displaymath}
	\tilde{W}\bigl(\xbf^{(s)},\ybf^{(t)}\bigr) \defs \bigl\{\ubf\in\Shuf(m,n)\mid \ubf_{\xbf}=\xbf^{(s)}\;\text{and}\;\ubf_{\ybf}=\ybf^{(t)}\bigr\}.
\end{displaymath}
We remark that $\tilde{W}\bigl(\xbf^{(s)},\ybf^{(t)}\bigr)$ does not depend on the concrete choice of subwords of $\xbf$ and $\ybf$, but only on the length of these words.  Since all elements in $\tilde{W}\bigl(\xbf^{(s)},\ybf^{(t)}\bigr)$ have the same support, Lemma~\ref{lem:bubble_order} implies that for $\ubf,\vbf\in\tilde{W}\bigl(\xbf^{(s)},\ybf^{(t)}\bigr)$ we have $\ubf\bubleq\vbf$ if and only if $\invset(\ubf)\subseteq\invset(\vbf)$.  We may thus embed the interval $\bigl(\tilde{W}\bigl(\xbf^{(s)},\ybf^{(t)}\bigr),\bubleq\bigr)$ into the \defn{weak order} on the set of all permutations of length $s+t$ by identifying the letter $x_{i}$ with the number $i$ and the letter $y_{j}$ with the number $s+j$.  See for instance \cite{markowsky:permutation} for more background on the weak order of permutations.  This perspective yields the following property of the interval $\bigl(\tilde{W}\bigl(\xbf^{(s)},\ybf^{(t)}\bigr),\bubleq\bigr)$.

\begin{proposition}\label{prop:bubble_same_support_distributive}
	For any $\xbf^{(s)}$ and $\ybf^{(t)}$, the induced subposet $\Bigl(\tilde{W}\bigl(\xbf^{(s)},\ybf^{(t)}\bigr),\bubleq\Bigr)$ is a distributive lattice with $\binom{s+t}{s}$ elements.
\end{proposition}
\begin{proof}
	Let $\tilde{W}=\tilde{W}\bigl(\xbf^{(s)},\ybf^{(t)}\bigr)$.  If $\wbf\bubleq\wbf'$ for $\wbf,\wbf'\in\tilde{W}$, then $\wbf'$ is obtained from $\wbf$ by a sequence of transpositions, because $\wbf$ and $\wbf'$ have the same support.  Moreover, $(\tilde{W},\bubleq)$ has a unique minimal element $\xbf^{(s)}\ybf^{(t)}$ and a unique maximal element $\ybf^{(t)}\xbf^{(s)}$. 
	
	Thus, if we consider the permutation 
	\begin{displaymath}
		\pi = s{+}1\;s{+}2\;\ldots s{+}t\;1\;2\;\ldots\;s, 
	\end{displaymath}
	then $(\tilde{W},\bubleq)$ is isomorphic to the weak order interval $\mathbf{I}_{s,t}$ between the identity permutation on $[s+t]$ and $\pi$.  By~\cite{stembridge:fully}*{Theorem~3.2}, this interval $\mathbf{I}_{s,t}$ is a distributive lattice, because $\pi$ is a \emph{fully commutative} permutation.  This claim and the cardinality result follow from \cite{stembridge:fully}*{Theorem~6.1}.
\end{proof}

\section{Lattice Properties of $\Bub(m,n)$}
	\label{sec:bubble_lattice}
\subsection{Joins in $\Bub(m,n)$}

We now prove the first part of Theorem~\ref{thm:bubble_lattice_main}, namely that the bubble poset is, in fact, a lattice.

\begin{theorem}\label{thm:bubble_lattice}
	For $m,n\geq 0$, the poset $\Bub(m,n)$ is a lattice.
\end{theorem}
\begin{proof}
	By Lemma~\ref{lem:bubble_duality}, meets in $\Bub(m,n)$ correspond to joins in $\Bub(n,m)$.  Therefore, it remains to establish the existence of joins in $\Bub(m,n)$.  

	For $\ubf,\vbf\in\Shuf(m,n)$ consider
	\begin{displaymath}
		W = \bigl\{\wbf\in\Shuf(m,n)\mid \ubf\bubleq\wbf\;\text{and}\;\vbf\bubleq\wbf\bigr\}.
	\end{displaymath}
	We want to show that $W$ has a unique minimal element. In order to achieve this, we define $\xbf'\defs \ubf_{\xbf}\cap\vbf_{\xbf}$ and $\ybf'\defs \ubf_{\ybf}\cup\vbf_{\ybf}$, and consider the set $\tilde{W}(\xbf',\ybf')$ defined before Proposition~\ref{prop:bubble_same_support_distributive}.  In fact, we write $\tilde{W}$ rather than $\tilde{W}(\xbf',\ybf')$.
	
	\medskip
	
	\underline{Claim 1}: Every minimal element of $(W,\bubleq)$ is contained in $\tilde{W}$.\\
	Let $\wbf\in W$ be minimal with respect to $\bubleq$.  Since $\ubf\bubleq\wbf$ and $\vbf\bubleq\wbf$, then Lemma~\ref{lem:bubble_order} implies that $\wbf_{\xbf}$ is a subword of $\xbf'$ and $\ybf'$ is a subword of $\wbf_{\ybf}$ and we have $\invset(\ubf_{\wbf})\subseteq\invset(\wbf_{\ubf})$ and $\invset(\vbf_{\wbf})\subseteq\invset(\wbf_{\vbf})$.  
	
	(a) Assume that there exists $x_{s}\in\xbf'\setminus\wbf_{\xbf}$.  By Corollary~\ref{cor:xy_removal}, there exists $\wbf'\bubcov\wbf$ which contains $x_{s}$, and necessarily $\wbf$ is obtained from $\wbf'$ by a right indel deleting $x_{s}$.  Moreover, $x_{s}$ is a letter of both $\ubf$ and $\vbf$, which implies that 
	$\wbf'_{\xbf}$ is a subword of $\xbf'$ and $\ybf'$ is a subword of $\wbf'_{\ybf}$.  Furthermore, $\invset(\ubf_{\wbf})$ is obtained from $\invset(\ubf_{\wbf'})$ (and likewise $\invset(\wbf_{\ubf})$ is obtained from $\invset(\wbf'_{\ubf})$) by deleting all inversions involving $x_{s}$.  Since $\wbf\in W$, we have $\ubf\bubleq\wbf$ and thus $\invset(\ubf_{\wbf})\subseteq\invset(\wbf_{\ubf})$.  Suppose that there is an inversion $(x_{s},y_{t})\in\invset(\ubf_{\wbf'})\setminus\invset(\wbf'_{\ubf})$.  This implies that $y_{t}$ is a letter of $\ubf$ and $\wbf'$, but $y_{t}$ appears before $x_{s}$ in $\ubf$, and after $x_{s}$ in $\wbf'$.  Since $\wbf'\sindel\wbf$, the next letter after $x_{s}$ in $\wbf'$ must be $x_{s'}$ for $s<s'$.  Since $x_{s'}\in\wbf$ we have $x_{s'}\in\ubf$.  Thus, $(x_{s'},y_{t})\in\invset(\ubf_{\wbf})$.  However, $(x_{s'},y_{t})\notin\invset(\wbf_{\ubf})$, because $x_{s'}$ must still come before $y_{t}$ in $\wbf$.  This is a contradiction, and we conclude that $\invset(\ubf_{\wbf'})\subseteq\invset(\wbf'_{\ubf})$, and we obtain $\ubf\bubleq\wbf'$.  The analogous reasoning shows that $\vbf\bubleq\wbf'$, and thus $\wbf'\in W$, contradicting the assumption that $\wbf$ is minimal.  
	
	(b) Now assume that there exists $y_{t}\in\wbf_{\ybf}\setminus\ybf'$.  Corollary~\ref{cor:xy_removal} yields two cases.  Either, there exists $\wbf'\bubcov\wbf$ which does not contain $y_{t}$ (then necessarily $\wbf$ is obtained from $\wbf'$ by a right indel inserting $y_{t}$) or there exists $\wbf''\bubcov\wbf$ with $\invset(\wbf)\setminus\invset(\wbf')=\bigl\{(x,y_{t})\bigr\}$.  In both cases, however, $y_{t}$ is neither a letter of $\ubf$ nor of $\vbf$.
	
	As in the previous case, we get that $\wbf'_{\xbf}$ is a subword of $\xbf'$ and $\ybf'$ is a subword of $\wbf'_{\ybf}$.  We see that $\ubf$ cannot contain an inversion involving $y_{t}$, because $y_{t}\notin\ubf$, and therefore $\invset(\ubf_{\wbf'})=\invset(\ubf_{\wbf})$.  The same argument shows that $\invset(\wbf'_{\ubf})=\invset(\wbf_{\ubf})$, because the only inversions of $\wbf$ that are not inversions of $\wbf'$ involve $y_{t}$.  This proves $\invset(\ubf_{\wbf'})=\invset(\ubf_{\wbf})\subseteq\invset(\wbf_{\ubf})=\invset(\wbf'_{\ubf})$, and therefore $\ubf\bubleq\wbf'$.  The analogous reasoning shows that $\vbf\bubleq\wbf'$, and thus $\wbf'\in W$, contradicting the assumption that $\wbf$ is minimal.
	
	For $\wbf''$, we observe that $\wbf$ and $\wbf''$ have the same support, so that $\wbf''_{\xbf}$ is a subword of $\xbf'$, $\ybf'$ is a subword of $\wbf''_{\ybf}$ and $\invset(\ubf_{\wbf})=\invset(\ubf_{\wbf''})$.  By Corollary~\ref{cor:xy_removal}, $\invset(\wbf'')=\invset(\wbf)\setminus\bigl\{(x,y_{t})\bigr\}$ for some letter $x$ immediately succeeding $y_{t}$ in $\wbf$.  Since $y_{t}\notin\ubf$, we conclude that $\invset(\wbf_{\ubf})=\invset(\wbf''_{\ubf})$.  Since $\ubf\bubleq\wbf$, we obtain $\invset(\ubf_{\wbf''})=\invset(\ubf_{\wbf})\subseteq\invset(\wbf_{\ubf})=\invset(\wbf''_{\ubf})$.  By Lemma~\ref{lem:bubble_order}, we get $\ubf\bubleq\wbf''$.  The same reasoning shows that $\vbf\bubleq\wbf''$ and therefore $\wbf''\in W$, contradicting the assumption that $\wbf$ is minimal.

	We conclude from (a) that $\wbf_{\xbf}=\xbf'$ and from (b) that $\wbf_{\ybf}=\ybf'$, which proves that $\wbf\in\tilde{W}$ as desired.
	
	\medskip
	
	\underline{Claim 2}: The intersection $W\cap\tilde{W}$ has a unique minimal element under $\bubleq$.\\
	Let $\wbf'$ and $\wbf''$ be two distinct minimal elements of $W\cap\tilde{W}$ with respect to $\bubleq$.  This means that $\wbf'$ and $\wbf''$ are incomparable with respect to $\bubleq$, meaning that there exist $x_{s'},x_{s''},y_{t'},y_{t''}$ such that $(x_{s'},y_{t'})\in\invset(\wbf')\setminus\invset(\wbf'')$ and $(x_{s''},y_{t''})\in\invset(\wbf'')\setminus\invset(\wbf')$.  In fact, these letters can be chosen such that $x_{s'}$ and $y_{t'}$ (and likewise $x_{s''}$ and $y_{t''}$) are adjacent in both $\wbf'$ and $\wbf''$.  
	
	Let $\wbf\in\tilde{W}$ be the element in which $x_{s'}$ appears immediately before $y_{t'}$ and $x_{s''}$ appears immediately before $y_{t''}$.  Then, $\wbf\bubcov\wbf'$ and $\wbf\bubcov\wbf''$.  Since $\wbf'$ and $\wbf''$ are minimal in $W\cap\tilde{W}$ it follows that $\wbf\notin W$.  Without loss of generality, we may assume that $\ubf\not\bubleq\wbf$.  Since all elements in $\tilde{W}$ have the same support, Lemma~\ref{lem:bubble_order} implies $\invset(\ubf_{\wbf})\not\subseteq\invset(\wbf_{\ubf})$.  However, since $\wbf\bubleq\wbf'$ and both words have the same set of letters, we have $\invset(\ubf_{\wbf})=\invset(\ubf_{\wbf'})\subseteq\invset(\wbf'_{\ubf})$ and $\invset(\wbf_{\ubf})\subseteq\invset(\wbf'_{\ubf})$.  So, if $(x_{s},y_{t})\in\invset(\ubf_{\wbf})\setminus\invset(\wbf_{\ubf})$, then $(x_{s},y_{t})\in\invset(\wbf'_{\ubf})\setminus\invset(\wbf_{\ubf})$.  Since, in fact, $\wbf\bubcov\wbf'$ it must be that $(x_{s},y_{t})=(x_{s'},y_{t'})$.  But, $(x_{s'},y_{t'})\notin\invset(\wbf'')$, which yields the contradiction $\ubf\not\bubleq\wbf''$.

	\medskip
	
	In conclusion, Claim 1 implies that every minimal element of $W$ lies in $\tilde{W}$, and Claim 2 implies that there is a unique minimal element in $W$, which must then be the join of $\ubf$ and $\vbf$.
\end{proof}

\begin{example}\label{ex:bubble_joins}
	Let $m=n=5$ and consider $\ubf=\rd{x_2x_4}\bl{y_1y_4}\rd{x_5}\bl{y_5}$ and $\vbf=\rd{x_3}\bl{y_1y_3}\rd{x_4x_5}$. Then, $\tilde{W}$ contains all words in $\Shuf(m,n)$ using the letters $\{\rd{x_4},\rd{x_5}\}$ and $\{\bl{y_1},\bl{y_3},\bl{y_4},\bl{y_5}\}$.  The minimal element in $\tilde{W}$ is clearly $\tilde{\wbf}=\rd{x_4x_5}\bl{y_1y_3y_4y_5}$, and we have
	\begin{align*}
		\invset(\ubf_{\tilde{\wbf}}) & = \bigl\{(\rd{x_5},\bl{y_1}),(\rd{x_5},\bl{y_4})\bigr\},\\
		\invset(\vbf_{\tilde{\wbf}}) & = \bigl\{(\rd{x_4},\bl{y_1}),(\rd{x_4},\bl{y_3}),(\rd{x_5},\bl{y_1}),(\rd{x_5},\bl{y_3})\bigr\}.
	\end{align*}
	So, for $\wbf\in\tilde{W}$ to be the join of $\ubf$ and $\vbf$ it must have all of these inversions, which is satisfied when
	\begin{displaymath}
		\ubf\vee\vbf = \wbf = \bl{y_1y_3}\rd{x_4}\bl{y_4}\rd{x_5}\bl{y_5}.
	\end{displaymath}
\end{example}

For $\ubf\in\Shuf(m,n)$ let $\ubf^{\ybf}$ denote the \defn{$\ybf$-filling} of $\ubf$, \ie the word obtained from $\ubf$ by inserting the letters from $\ybf$ which are not present in $\ubf$ as rightmost as possible without changing the relative order of the letters present in $\ubf$.  More precisely, suppose that $\ybf\setminus\ubf_{\ybf}=\{y_{j_{1}},y_{j_{2}},\ldots,y_{j_{s}}\}$ where $j_{1}>j_{2}>\cdots>j_{s}$.  We set $\ubf^{(0)}=\ubf$ and for $i\in[s]$ we create $\ubf^{(i)}$ from $\ubf^{(i-1)}$ by inserting $y_{j_{i}}$ immediately left of $y_{t_{i}}$, where
\begin{displaymath}
	t_{i} = \min\{k\mid k>j_{i} \;\text{and}\;y_{k}\in\ubf^{(i-1)}\}.
\end{displaymath}
When $\ubf^{(i-1)}$ does not contain $y_{k}$ for $k>j_{i}$, then we add $y_{j_{i}}$ at the end of $\ubf^{(i-1)}$.  Finally, we set $\ubf^{\ybf}\defs\ubf^{(s)}$.  Dually, we define $\ubf^{\xbf}$ by inserting the missing letters of $\xbf$ in the analogous manner.

\begin{example}
	Let $m=5$ and $n=6$ and consider $\ubf=\rd{x_2x_3}\bl{y_1y_3}\rd{x_4}\bl{y_5}\rd{x_5}$.  Then $\ybf\setminus\ubf_{\ybf}=\{\bl{y_6},\bl{y_4},\bl{y_2}\}$.  With $\ubf^{(0)}=\ubf$, we get
	\begin{align*}
		\ubf^{(1)} & = \rd{x_2x_3}\bl{y_1y_3}\rd{x_4}\bl{y_5}\rd{x_5}\bl{y_6},\\
		\ubf^{(2)} & = \rd{x_2x_3}\bl{y_1y_3}\rd{x_4}\bl{y_4y_5}\rd{x_5}\bl{y_6},\\
		\ubf^{(3)} & = \rd{x_2x_3}\bl{y_1y_2y_3}\rd{x_4}\bl{y_5}\rd{x_5}\bl{y_6}.
	\end{align*}
	Thus $\ubf^{\ybf}=\ubf^{(3)}\in\Shuf(5,6)$.
\end{example}

A \defn{closure operator} on a poset $(P,\leq)$ is a map $f\colon P\to P$ which is
\begin{itemize}
	\item \defn{idempotent}: $f(f(p))=f(p)$ for all $p\in P$,
	\item \defn{extensive}: $p\leq f(p)$, and 
	\item \defn{monotone}: $p\leq q$ implies $f(p)\leq f(q)$.
\end{itemize}
The fixed points of $f$ are called \defn{closed}.

\begin{lemma}\label{lem:bubble_closure}
	The map $\ubf\mapsto\ubf^{\ybf}$ is a closure operator on $\Bub(m,n)$.
\end{lemma}
\begin{proof}
	Since $\ubf^{\ybf}$ contains all letters of $\ybf$ it is clear that $(\ubf^{\ybf})^{\ybf}=\ubf^{\ybf}$.  Moreover, $\ubf^{\ybf}$ is obtained from $\ubf$ by adding letters from $Y$ without changing or deleting any of the existing inversions.  Thus, $\ubf\bubleq\ubf^{\ybf}$ by Lemma~\ref{lem:bubble_order}.  This takes care of idempotence and extensivity.
	
	Now suppose that $\ubf\bubleq\vbf$, and let us write $\tilde{\ubf}=\ubf^{\ybf}$ and $\tilde{\vbf}=\vbf^{\ybf}$.  Since $\tilde{\ubf}_{\xbf}=\ubf_{\xbf}$ and $\tilde{\vbf}_{\xbf}=\vbf_{\xbf}$ it is clear that $\tilde{\ubf}_{\xbf}$ is a subword of $\tilde{\vbf}_{\xbf}$ because $\ubf_{\xbf}$ is a subword of $\vbf_{\xbf}$ by Lemma~\ref{lem:bubble_order}.  Moreover $\tilde{\ubf}_{\ybf}=\ybf=\tilde{\vbf}_{\ybf}$.  
	
	Choose an inversion $(x_{s},y_{t})\in\invset(\tilde{\ubf}_{\tilde{\vbf}})$.  This implies immediately that $x_{s}$ is a letter of $\tilde{\vbf}$ and thus also of $\vbf$.  If this is an inversion of $\ubf_{\vbf}$, then by Lemma~\ref{lem:bubble_order}, $(x_{s},y_{t})\in\invset(\vbf_{\ubf})\subseteq\invset(\tilde{\vbf}_{\tilde{\ubf}})$.  Otherwise, $y_{t}\notin\ubf$.  By construction, $y_{t}$ can only be inserted before $x_{s}$ if there is some $y_{t'}\in\ubf$ with $t'>t$ that comes before $x_{s}$ in $\ubf$. Now, since $\ubf\bubleq\vbf$, $\vbf$ contains all letters of $\ybf$ that are in $\ubf$; in particular $y_{t'}\in\vbf$.  We have already noted that $x_{s}\in\vbf$ which means that $y_{t}$ appears before $x_{s}$ in $\tilde{\vbf}$.  Thus, $(x_{s},y_{t})\in\invset(\tilde{\vbf}_{\tilde{\ubf}})$.  We have thus established monotonicity.
\end{proof}

\begin{corollary}
	A word $\ubf\in\Shuf(m,n)$ is closed if and only if $\ubf_{\ybf}=\ybf$.
\end{corollary}

We may use the previously defined $\ybf$-filling to characterize the joins in $\Bub(m,n)$.

\begin{lemma}\label{lem:bubble_joins_fill}
	Let $\ubf,\vbf\in\Shuf(m,n)$, and define $\tilde{\wbf}=x_{j_{1}}x_{j_{2}}\cdots x_{j_{s}} y_{k_{1}}y_{k_{2}}\cdots y_{k_{t}}$, where $\ubf_{\xbf}\cap\vbf_{\xbf}=\{x_{j_{1}},x_{j_{2}},\ldots,x_{j_{s}}\}$ and $\ubf_{\ybf}\cup\vbf_{\ybf}=\{y_{k_{1}},y_{k_{2}},\ldots,y_{k_{t}}\}$.  The join $\wbf=\ubf\vee\vbf$ in $\Bub(m,n)$ is uniquely determined by
	\begin{itemize}
		\item $\wbf_{\xbf} = \tilde{\wbf}_{\xbf}$,
		\item $\wbf_{\ybf} = \tilde{\wbf}_{\ybf}$, and
		\item $\invset(\wbf) = \invset\bigl((\ubf^{\ybf})_{\tilde{\wbf}}\bigr)\cup\invset\bigl((\vbf^{\ybf})_{\tilde{\wbf}}\bigr)$.
	\end{itemize}
\end{lemma}
\begin{proof}
	Let $\tilde{W} = \bigl\{\wbf'\in\Shuf(m,n)\mid \wbf'_{\xbf}=\tilde{\wbf}_{\xbf}\;\text{and}\;\wbf'_{\ybf}=\tilde{\wbf}_{\ybf}\bigr\}$.  This is the same set that was used in the proof of Theorem~\ref{thm:bubble_lattice}.  We abbreviate $\tilde{\ubf}=(\ubf^{\ybf})_{\tilde{\wbf}}$ and $\tilde{\vbf}=(\vbf^{\ybf})_{\tilde{\wbf}}$.  Then, $\tilde{\ubf},\tilde{\vbf}\in\tilde{W}$.  By Lemma~\ref{lem:bubble_closure}, $\ubf\bubleq\ubf^{\ybf}$ and $\vbf\bubleq\vbf^{\ybf}$.  Since $\tilde{\ubf}$ is obtained from $\ubf^{\ybf}$ by deleting some $x$'s and some $y$'s together with their corresponding inversions, it follows that $\ubf\bubleq\tilde{\ubf}\bubleq\ubf^{\ybf}$ and $\tilde{\ubf}\in\tilde{W}$.
	In fact, $\tilde{\ubf}$ is the smallest element in $\tilde{W}$ which is above $\ubf$.   The analogous property holds for $\tilde{\vbf}$.
	
	Since $\tilde{\ubf}$ and $\tilde{\vbf}$ have the same set of letters, we obtain $\tilde{\ubf}\vee\tilde{\vbf}\in\tilde{W}$ by adding the inversions from $\tilde{\vbf}$ to those of $\tilde{\ubf}$ (or vice versa).  By construction of $\tilde{\wbf}$, we get that $\wbf=\tilde{\ubf}\vee\tilde{\vbf}$.  We have thus shown that $\wbf$ has the desired properties.
\end{proof}

\begin{example}
	Consider again $\ubf=\rd{x_2x_4}\bl{y_1y_4}\rd{x_5}\bl{y_5}$ and $\vbf=\rd{x_3}\bl{y_1y_3}\rd{x_4x_5}$ as elements of $\Shuf(5,5)$.  We get $\tilde{\wbf}=\rd{x_4x_5}\bl{y_1y_3y_4y_5}$.  Filling $\ubf$ and $\vbf$ with the missing letters of $\ybf$ and then restricting to $\tilde{\wbf}$ yields:
	\begin{align*}
		\ubf^{\ybf} & = \rd{x_2x_4}\bl{y_1y_2y_3y_4}\rd{x_5}\bl{y_5}, & \ubf^{\ybf}_{\tilde{\wbf}} & = \rd{x_4}\bl{y_1y_3y_4}\rd{x_5}\bl{y_5},\\
		\vbf^{\ybf} & = \rd{x_3}\bl{y_1y_2y_3}\rd{x_4x_5}\bl{y_4y_5}, & \vbf^{\ybf}_{\tilde{\wbf}} & = \bl{y_1y_3}\rd{x_4x_5}\bl{y_4y_5}.
	\end{align*}
	We have
	\begin{align*}
		\invset(\ubf^{\ybf}_{\tilde{\wbf}}) & = \bigl\{(\rd{x_5},\bl{y_{1}}),(\rd{x_5},\bl{y_3}),(\rd{x_5},\bl{y_4})\bigr\},\\
		\invset(\vbf^{\ybf}_{\tilde{\wbf}}) & = \bigl\{(\rd{x_4},\bl{y_{1}}),(\rd{x_4},\bl{y_3}),(\rd{x_5},\bl{y_1}),(\rd{x_5},\bl{y_3})\bigr\},
	\end{align*}
	and therefore $\invset(\wbf)=\invset(\ubf^{\ybf}_{\tilde{\wbf}})\cup\invset(\vbf^{\ybf}_{\tilde{\wbf}})$, where $\wbf=\bl{y_1y_3}\rd{x_4}\bl{y_4}\rd{x_5}\bl{y_5}$ is the join of $\ubf$ and $\vbf$ computed in Example~\ref{ex:bubble_joins}.
\end{example}

\subsection{Extremality of $\Bub(m,n)$}

We now prove some further lattice-theoretic properties of $\Bub(m,n)$.  We start by establishing extremality, and first describe the join-irreducible elements of $\Bub(m,n)$.

\begin{lemma}\label{lem:bubble_irreducibles}
	Let $\ubf\in\Shuf(m,n)$.  We have $\ubf\in\JI\bigl(\Bub(m,n)\bigr)$ if and only if either 
	$\ubf=\xbf_{\hat{\imath}}$ for some $i\in[m]$ or $\ubf_{\xbf}=\xbf$ and $\ubf_{\ybf}=y_{j}$ for some $j\in[n]$.
\end{lemma}
\begin{proof}
	This is immediate from Corollary~\ref{cor:xy_removal}, because every letter of $\xbf$ not present in $\ubf$ and every letter of $\ybf$ present in $\ubf$ accounts for a lower cover.  So $\ubf$ is join-irreducible if and only if either $\ubf$ consists of all but one letter of $\xbf$ and none of $\ybf$, or it consists of all letters of $\xbf$ and precisely one letter of $\ybf$.  There are clearly $m$ possible join irreducibles of the first type and $n(m+1)$ of the second, because there are $m+1$ positions where the letter of $\ybf$ can be inserted into $\xbf$.
\end{proof}

\begin{corollary}\label{cor:bubble_irreducibles_count}
	For $m,n\geq 0$ we have 
	\begin{displaymath}
		\bigl\lvert\JI\bigl(\Bub(m,n)\bigr)\bigr\rvert=mn+m+n=\bigl\lvert\MI\bigl(\Bub(m,n)\bigr)\bigr\rvert.
	\end{displaymath}
\end{corollary}

\begin{corollary}\label{cor:bubble_irreducibles_poset}
	The poset of join-irreducibles, $\bigl(\JI\bigl(\Bub(m,n)\bigr),\bubleq\bigr)$, of $\Bub(m,n)$ is a disjoint union of an $m$-antichain and $n$-many $m+1$ chains.
\end{corollary}
\begin{proof}
	The $m$-antichain consists of all the words $\xbf_{\hat{\imath}}$ for $i\in[m]$.  The $i$-th chain consists of the words with support $\{x_{1},x_{2},\ldots,x_{m},y_{i}\}$.  Clearly, the smallest element in this chain is $\xbf y_{i}$ and the greatest one is $y_{i}\xbf$.  Thus, if we set $\ubf^{(0)}=\xbf$, $\ubf^{(1)}=\xbf y_{i}$ and $\ubf^{(j+1)}$ is the join-irreducible shuffle word in which $y_{i}$ appears immediately before $x_{m+1-j}$ for $j\in[m]$, then we get $\ubf^{(0)}\bubcov\ubf^{(1)}\bubcov\ubf^{(2)}\bubcov\cdots\bubcov\ubf^{(m+1)}$ and $\ubf^{(1)}$ is obtained from $\ubf^{(0)}$ by inserting $y_{i}$ and $\invset(\ubf^{(j+1)})\setminus\invset(\ubf^{(j)})=\bigl\{(x_{m+1-j},y_{i})\bigr\}$.
\end{proof}

\begin{theorem}\label{thm:bubble_extremal}
	For $m,n\geq 0$, the lattice $\Bub(m,n)$ is extremal.
\end{theorem}
\begin{proof}
	By Lemma~\ref{lem:bubble_irreducibles}, it remains to show that $\Bub(m,n)$ has length $(m+1)n+m$.  Let $\ubf^{(0)}\defs\xbf$ and for $i\in[m]$ we set $\ubf^{(i)}\defs \ubf^{(i-1)}y_{i}$.  Then, we set $\wbf^{(0,1)}\defs\xbf\ybf=\ubf^{(n)}$.  It is straightforward to check that $\ubf^{(i-1)}\sindel\ubf^{(i)}$ and thus $\ubf^{(i-1)}\bubcov\ubf^{(i)}$ by Lemma~\ref{lem:bubble_covers}.
	
	For $i\in[m]$ we construct $\wbf^{(i,1)}$ from $\wbf^{(i-1,1)}$ by transposing $x_{m+i-1}$ with $y_{1}$.  Then, $\wbf^{(m,1)}=y_{1}\xbf y_{2}y_{3}\cdots y_{n}$.  Inductively, we set $\wbf^{(0,j)}\defs\wbf^{(m,j-1)}$ for $j>1$, and we construct $\wbf^{(i,j)}$ from $\wbf^{(i-1,j)}$ by transposing $x_{m+1-i}$ with $y_{j}$.  Then, we have $\wbf^{(i-1,j)}\transpose\wbf^{(i,j)}$ and thus $\wbf^{(i-1,j)}\bubcov\wbf^{(i,j)}$ by Lemma~\ref{lem:bubble_covers}.
	
	Finally, we set $\vbf^{(0)}\defs\ybf\xbf=\wbf^{(m,n)}$, and for $j\in[n]$ we construct $\vbf^{(j)}$ from $\vbf^{(j-1)}$ by deleting its last letter (which is $x_{m-i+1}$).  Then, $\vbf^{(j-1)}\sindel\vbf^{(j)}$ and thus $\vbf^{(j-1)}\bubcov\vbf^{(j)}$ by Lemma~\ref{lem:bubble_covers}.  Since $\vbf^{(n)}=\ybf$ we have constructed a maximal chain of $\Bub(m,n)$ with $mn+m+n+1$ elements.  Hence, the length of $\Bub(m,n)$ is at least $mn+m+n$.  However, \eqref{eq:lattice_length_irreds} and Corollary~\ref{cor:bubble_irreducibles_count} imply that the length of $\Bub(m,n)$ cannot exceed $mn+m+n$, which completes the proof.
\end{proof}

\begin{remark}
	The maximal chain constructed in the proof of Theorem~\ref{thm:bubble_extremal} first adds all letters of $\ybf$ to $\xbf$ in increasing order, then, for $i\in[n]$, it transposes $y_{i}$ all the way across $\xbf$, and then deletes all letters of $\xbf$ in decreasing order.
\end{remark}

\subsection{Interval-Constructability of $\Bub(m,n)$}

\begin{figure}
	\centering
	\includegraphics[page=6,scale=1]{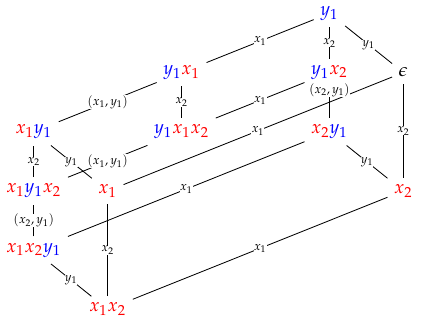}
	\caption{The poset $\Bub(2,1)$ labeled by $\lambda$.}
	\label{fig:bubble_21_labeled}
\end{figure}

Using the description of the covering pairs in $\Bub(m,n)$ in Corollary~\ref{cor:xy_removal}, we define the following edge labeling of $\Bub(m,n)$:
\begin{equation}\label{eq:bubble_labeling}
	\lambda(\ubf,\vbf) \defs 
	\begin{cases}
		x, & \text{if}\;\ubf\sindel\vbf, \ubf_{\ybf}=\vbf_{\ybf},\ubf_{\xbf}\setminus\vbf_{\xbf}=\{x\},\\
		y, & \text{if}\;\ubf\sindel\vbf, \ubf_{\xbf}=\vbf_{\xbf},\vbf_{\ybf}\setminus\ubf_{\ybf}=\{y\},\\
		(x,y), & \text{if}\;\ubf\transpose\vbf, \invset(\vbf)\setminus\invset(\ubf)=\bigl\{(x,y)\bigr\}.
	\end{cases}
\end{equation}
We prove in this section that this labeling is in fact a CU-labeling of $\Bub(m,n)$.  Figure~\ref{fig:bubble_21_labeled} shows the lattice $\Bub(2,1)$ with this labeling.  We start by investigating the polygonal intervals of $\Bub(m,n)$.

\begin{lemma}
	The polygonal intervals of $\Bub(m,n)$ consist either of four or of five elements.
\end{lemma}
\begin{proof}
	Let $[\ubf,\vbf]$ be a polygonal interval and let $\ubf'$ and $\ubf''$ denote the two upper covers of $\ubf$ in that interval.  We distinguish several cases.
	
	(i) Let $\lambda(\ubf,\ubf')=x_{s'}$ and $\lambda(\ubf,\ubf'')=x_{s''}$.  Then, $\ubf'$ is obtained by deleting $x_{s'}$ and $\ubf''$ is obtained by deleting $x_{s''}$.  Thus, the letter immediately after $x_{s'}$ and $x_{s''}$ is not in $Y$ and is still present in $\ubf'$ and $\ubf''$.  This means that we can delete $x_{s''}$ in $\ubf'$ and $x_{s'}$ in $\ubf''$ and obtain the same element each way.  This means $\ubf'\bubcov\vbf$ and $\ubf''\bubcov\vbf$ such that $\lambda(\ubf',\vbf)=x_{s''}$ and $\lambda(\ubf'',\vbf)=x_{s'}$.
	
	(ii) Let $\lambda(\ubf,\ubf')=y_{t'}$ and $\lambda(\ubf,\ubf'')=y_{t''}$.  Then, $\ubf'$ is obtained by inserting $y_{t'}$ and $\ubf''$ is obtained by inserting $y_{t''}$.  Without loss of generality we may assume that $t'<t''$.  If both $y_{t'}$ and $y_{t''}$ are inserted at the end of $\ubf$ (or immediately before some letter $y$), then we may insert $y_{t''}$ at the end of $\ubf'$ (or immediately before $y$) and we may insert $y_{t'}$ immediately before $y_{t''}$ in $\ubf''$.  In both cases, this yields the same element, and we thus get $\ubf'\bubcov\vbf$ and $\ubf''\bubcov\vbf$ such that $\lambda(\ubf',\vbf)=y_{t''}$ and $\lambda(\ubf'',\vbf)=y_{t'}$.  If $y_{t'}$ is inserted immediately before some letter $y$ in $\ubf$ and $y_{t''}$ is inserted at the end of $\ubf$ (or immediately before some different letter $\tilde{y}$), then we may insert $y_{t''}$ in $\ubf'$ at the end (or before the letter $\tilde{y}$), and we may insert $y'$ in $\ubf''$ before the letter $y$.  Once again, we get the same element by both insertions, and we get $\ubf'\bubcov\vbf$ and $\ubf''\bubcov\vbf$ such that $\lambda(\ubf',\vbf)=y_{t''}$ and $\lambda(\ubf'',\vbf)=y_{t'}$.
	
	(iii) Let $\lambda(\ubf,\ubf')=x_{s'}$ and $\lambda(\ubf,\ubf'')=y_{t''}$.  Then, $\ubf'$ is obtained by deleting $x_{s'}$, which means that there is some letter $x$ immediately right of $x_{s'}$ in $\ubf$.  Moreover, $\ubf''$ is obtained by inserting $y_{t''}$ either at the end of $\ubf$ or before some letter $y$.  These operations do not interfere, so that we can still insert $y_{t''}$ into $\ubf'$ and delete $x_{s'}$ from $\ubf''$, and obtain the same element either way.  We thus get $\ubf'\bubcov\vbf$ and $\ubf''\bubcov\vbf$ such that $\lambda(\ubf',\vbf)=y_{t''}$ and $\lambda(\ubf'',\vbf)=x_{s'}$.
		
	(iv) Let $\lambda(\ubf,\ubf')=x_{s'}$ and $\lambda(\ubf,\ubf'')=(x_{s''},y_{t''})$.  By construction, the letter right after $x_{s'}$ in $\ubf$ is in $X$, while the letter right after $x_{s''}$ is $y_{t''}$.  In particular, $s'\neq s''$.  If $x\neq x_{s''}$, then the corresponding operations do not interfere, meaning that $y_{t''}$ is still right after $x_{s''}$ in $\ubf'$ and $x$ is still next to $x_{s'}$ in $\ubf''$.  Therefore, we can still transpose $x_{s''}$ and $y_{t''}$ in $\ubf'$ and we can still delete $x_{s'}$ from $\ubf''$.  Once again, this yields the same element, and we get $\ubf'\bubcov\vbf$ and $\ubf''\bubcov\vbf$ such that $\lambda(\ubf',\vbf)=(x_{s''},y_{t''})$ and $\lambda(\ubf'',\vbf)=x_{s'}$.  If $x=x_{s''}$, then we can still transpose $x_{s''}$ and $y_{t''}$ in $\ubf'$ to obtain some element $\tilde{\vbf}$. However, we can no longer delete $x_{s'}$ from $\ubf''$, because the letter right after $x_{s'}$ in $\ubf''$ is now $y_{t''}$.  However, we can transpose $x_{s'}$ and $y_{t''}$ to obtain an element $\wbf$ in which $x_{s'}$ and $x_{s''}$ are adjacent again.  We may thus delete $x_{s'}$ in $\wbf$, and obtain some element $\tilde{\wbf}$.  But then, $\tilde{\vbf}=\tilde{\wbf}$, because in comparison to $\ubf$, none of these words contains $x_{s'}$ and both have the inversion $(x_{s''},y_{t''})$.  Since the relative positions of all other entries are the same as in $\ubf$, we conclude that the must be equal.  In fact, these two words are equal to $\vbf$, which yields a polygonal interval with five elements given by the labels $\lambda(\ubf',\vbf)=(x_{s''},y_{t''})$, $\lambda(\ubf'',\wbf)=(x_{s'},y_{t''})$ and $\lambda(\wbf,\vbf)=x_{s'}$.
	
	(v) Let $\lambda(\ubf,\ubf')=y_{t'}$ and $\lambda(\ubf,\ubf'')=(x_{s''},y_{t''})$.  If $y_{t''}$ is not the smallest letter in $\ubf$ which is greater than $y_{t'}$, then these two operations do not interfere, and we may insert $y_{t'}$ into $\ubf''$ and we may transpose $x_{x''}$ and $y_{t''}$ in $\ubf'$ to obtain the same element $\vbf$ satisfying $\ubf'\bubcov\vbf$ and $\ubf''\bubcov\vbf$ such that $\lambda(\ubf',\vbf)=(x_{s''},y_{t''})$ and $\lambda(\ubf'',\vbf)=y_{t'}$.  If $y_{t''}$ is the smallest letter in $\ubf$ which is greater than $y_{t'}$, then $y_{t'}$ is inserted right before $y_{t''}$.  This means that we cannot transpose $x_{s''}$ and $y_{t''}$ in $\ubf'$.  We may, however, transpose $x_{s''}$ and $y_{t'}$ to obtain an element $\tilde{w}$ in which $x_{s''}$ comes immediately before $y_{t''}$ so that we can transpose these two elements to obtain some element $\tilde{\wbf}$.  In $\ubf''$, we may still insert $y_{t'}$ immediately before $y_{t''}$ and we obtain an element $\tilde{\vbf}$.  Once again, we have $\tilde{\vbf}=\tilde{\wbf}$, because comparing to $\ubf$ we have the additional letter $y_{t'}$ and the additional inversion $(x_{s''},y_{t''})$ in both cases, and the relative positions of all other entries are the same. As before, these two words are equal to $\vbf$, so that we get a polygonal interval with five elements given by the labels $\lambda(\ubf',\wbf)=(x_{s''},y_{t'})$, $\lambda(\wbf,\vbf)=(x_{s''},y_{t''})$ and $\lambda(\ubf'',\vbf)=y_{t'}$.
	
	(vi) Let $\lambda(\ubf,\ubf')=(x_{s'},y_{t'})$ and $\lambda(\ubf,\ubf'')=(x_{s''},y_{t''})$.  In $\ubf$, $x_{s'}$ is immediately before $y_{t'}$ and $x_{s''}$ is immediately before $y_{t''}$.  Thus, in $\ubf'$ we can still transpose $x_{s''}$ and $y_{t''}$ and in $\ubf''$ we can still transpose $x_{s'}$ and $y_{t'}$, and obtain the same element $\vbf$ each time.  Therefore, we have $\lambda(\ubf',\vbf)=(x_{s''},y_{t''})$ and $\lambda(\ubf'',\vbf)=(x_{s'},y_{t'})$.
	
	In summary, we may obtain polygonal intervals with five elements from (iv) or (v); all other cases yield polygonal intervals with four elements.
\end{proof}

For later use, let us extract the following properties of $\lambda$ when restricted to polygonal intervals.

\begin{corollary}\label{cor:label_properties}
	Let $\ubf,\vbf\in\Shuf(m,n)$ be such that $[\ubf,\vbf]$ is a polygonal interval in $\Bub(m,n)$.  Moreover, let $\ubf',\ubf''$ be the unique upper covers of $\ubf$ and let $\vbf',\vbf''$ be the unique lower covers of $\vbf$ in $[\ubf,\vbf]$.  Let $C',C''$ be the two maximal chains of $[\ubf,\vbf]$, so that $\ubf',\vbf'\in C'$ and $\ubf'',\vbf''\in C''$.
	\begin{enumerate}
		\item[\rm (i)] $\lambda(\ubf,\ubf')=\lambda(\vbf'',\vbf)$ and $\lambda(\ubf,\ubf'')=\lambda(\vbf',\vbf)$.
		\item[\rm (ii)] The label sequences of $C'$ and $C''$ contain no duplicates.
	\end{enumerate}
	Suppose that $[\ubf,\vbf]$ has five elements, where without loss of generality $\ubf''=\vbf''$.  Then,
	\begin{enumerate}
		\item[\rm (iii)] either $\lambda(\ubf,\ubf')=(x_{s'},y_{t'})$, $\lambda(\ubf',\vbf')=(x_{s''},y_{t'})$, $\lambda(\vbf',\vbf)=x_{s'}$ for $s''<s'$, 
		\item[\rm (iv)] or $\lambda(\ubf,\ubf')=y_{t''}$, $\lambda(\ubf',\vbf')=(x_{s'},y_{t''})$, $\lambda(\vbf',\vbf)=(x_{s'},y_{t'})$ for $t''<t'$.
	\end{enumerate}
\end{corollary}

Let us consider the partial order $\prec$ on $X\uplus Y\uplus(X\times Y)$ defined by $x_{s}\prec(x_{s},y_{t})$ for all $t\in[n]$, $y_{t}\prec(x_{s},y_{t})$ for all $s\in[m]$, $(x_{s},y_{t})\prec (x_{s'},y_{t})$ when $s>s'$ and $(x_{s},y_{t})\prec(x_{s},y_{t'})$ when $t>t''$, and then closed by reflexivity and transitivity.  Figure~\ref{fig:shard_order_43} shows the resulting poset $\mathbf{S}_{m,n}$ for $m=4$ and $n=3$.  

\begin{figure}
	\centering
	\includegraphics[page=7,scale=1]{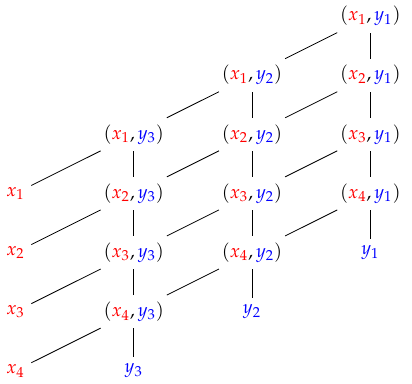}
	\caption{The poset $\mathbf{S}_{4,3}$.}
	\label{fig:shard_order_43}
\end{figure}

\begin{proposition}\label{prop:bubble_cu_labeling}
	The labeling $\lambda$ defined in \eqref{eq:bubble_labeling} is a CU-labeling of $\Bub(m,n)$ with respect to the poset $\mathbf{S}_{m,n}$.
\end{proposition}
\begin{proof}
	\eqref{it:cu1} follows from Corollary~\ref{cor:label_properties}(i) and \eqref{it:cu3} follows from Corollary~\ref{cor:label_properties}(ii).  The definition of $\mathbf{S}_{m,n}$ together with Corollary~\ref{cor:label_properties}(iv)-(v) implies \eqref{it:cu2}.  Equation~\eqref{it:cu4} follows from Lemma~\ref{lem:bubble_irreducibles}, and \eqref{it:cu5} follows from the same lemma together with Lemma~\ref{lem:bubble_duality}.
\end{proof}

\begin{theorem}\label{thm:bubble_congruence_uniform}
	For $m,n\geq 0$, the lattice $\Bub(m,n)$ is constructable by interval doublings.
\end{theorem}
\begin{proof}
	This follows from Theorem~\ref{thm:congruence_uniform_labeling} and Proposition~\ref{prop:bubble_cu_labeling}.
\end{proof}

\begin{corollary}\label{cor:bubble_semidistributive_trim}
	For $m,n\geq 0$, the lattice $\Bub(m,n)$ is semidistributive and trim.
\end{corollary}
\begin{proof}
	The first claim follows from Theorems~\ref{thm:congruence_uniform_semidistributive} and \ref{thm:bubble_congruence_uniform}, and the second claim follows from the first and Theorems~\ref{thm:semidistributive_extremal_trim} and \ref{thm:bubble_extremal}.
\end{proof}

We conclude this section by assembling the proof of Theorem~\ref{thm:bubble_lattice_main}.

\begin{proof}[Proof of Theorem~\ref{thm:bubble_lattice_main}]
	The fact that $\Bub(m,n)$ is a lattice is Theorem~\ref{thm:bubble_lattice}.  The constructability by interval doublings is Theorem~\ref{thm:bubble_congruence_uniform} and the semidistributivity and trimness were shown in Corollary~\ref{cor:bubble_semidistributive_trim}.
\end{proof}

Figure~\ref{fig:bubble_22} shows $\Bub(2,2)$.  It is laid out in such a way that it is possible to reconstruct the doubling procedure.

\begin{figure}
	\centering
	\includegraphics[page=1,scale=1]{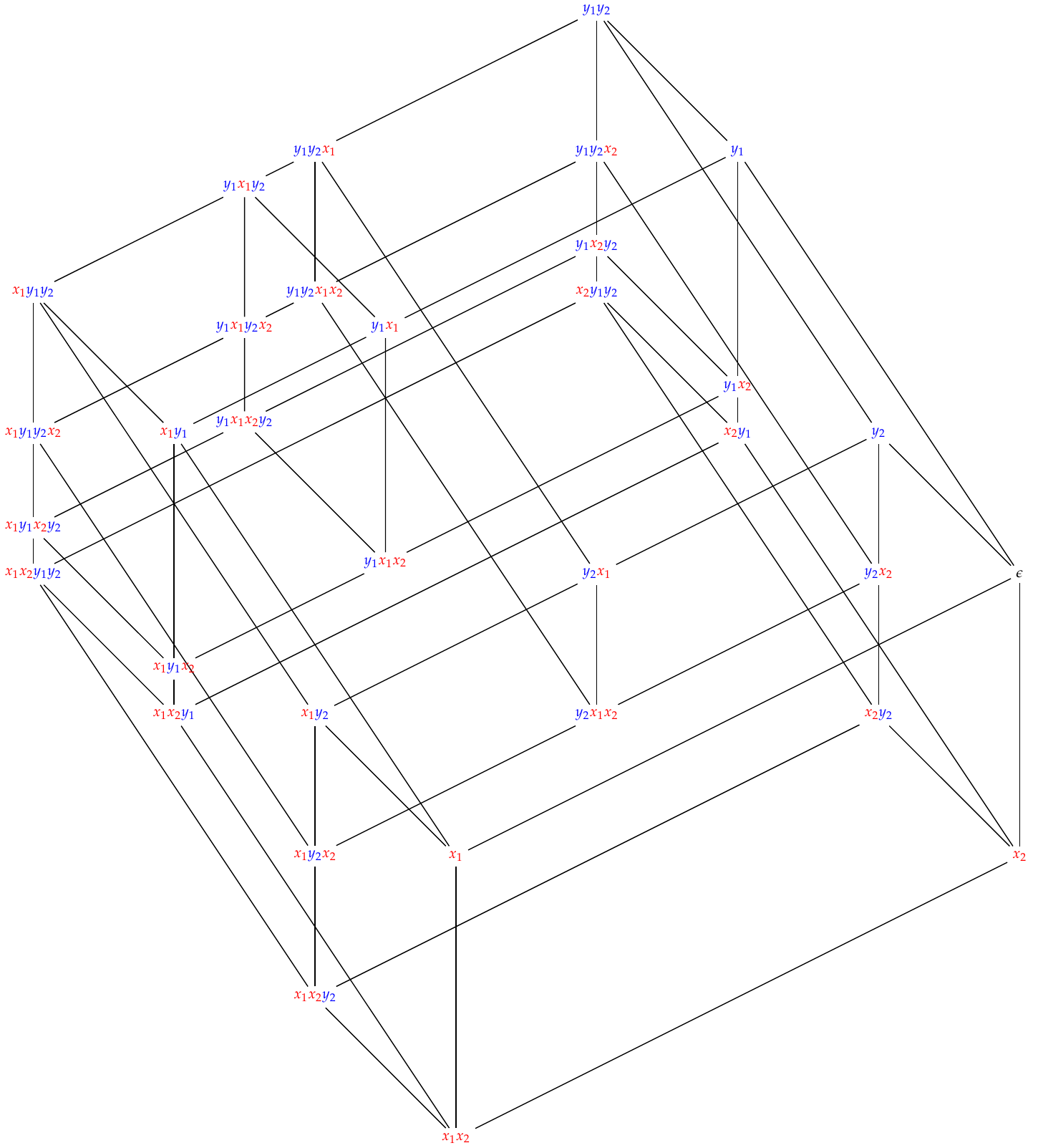}
	\caption{The bubble lattice $\Bub(2,2)$.}
	\label{fig:bubble_22}
\end{figure}

\subsection{Relation to Hochschild Lattices}
	\label{sec:hochschild}
One of the main motivations for the research presented here comes from an intriguing connection between certain shuffle lattices and the so-called \defn{Hochschild lattices} observed in \cite{muehle:hochschild}.  The Hochschild lattice $\Hoch(n)$ was initially defined as a certain interval in a certain partial order on Dyck paths~\cite{chapoton:dyck}, and it was later realized as the component order on certain integer tuples, called triwords~\cite{combe:geometric}.  More precisely, a \defn{triword} of length $n$ is an integer tuple $\mathfrak{u}=(u_{1},u_{2},\ldots,u_{n})$ with the following properties:
\begin{description}
	\item[T1\label{it:tri_1}] $u_{i}\in\{0,1,2\}$ for $i\in[n]$,
	\item[T2\label{it:tri_2}] $u_{1}\neq 2$,
	\item[T3\label{it:tri_3}] if $u_{i}=0$, then $u_{j}\neq 1$ for all $j>i$.
\end{description}
We write $\Tri(n)$ for the set of all triwords of length $n$ and $\compleq$ for the componentwise order on integer tuples.  The \defn{Hochschild lattice} is the poset $\Hoch(n)\defs\bigl(\Tri(n),\compleq\bigr)$.  This name is justified by the following result.

\begin{theorem}[\cite{combe:geometric}*{Theorem~2.3~and~Proposition~3.2}]\label{thm:hochschild_lattice}
	For $n>0$, the poset $\Hoch(n)$ is an extremal, interval-constructable lattice.
\end{theorem}

In \cite{muehle:hochschild}*{Theorem~1.3}, the second author showed that $\ShufPoset(n-1,1)$ is isomorphic to a certain reordering of $\Hoch(n)$.  We close the circle by proving that our bubble lattice $\Bub(n-1,1)$ is isomorphic to $\Hoch(n)$.  

To that end, we convert a shuffle word $\ubf\in\Shuf(n-1,1)$ to an integer tuple $\tilde\sigma(\ubf)=(u_{1},u_{2},\ldots,u_{n})$ as follows: if $x_{s}\notin\ubf_{\xbf}$, then $u_{n+1-s}=2$.  If $y_{1}$ is immediately after $x_{s}$, then we set $u_{i}=1$ for all $i\in[n-s]$ with $u_{i}\neq 2$ (if $y_{1}$ is the first letter, then this is to be interpreted as $s=0$).  The remaining entries are set to $0$.  For instance, for $\ubf=\rd{x_2x_4x_5}\bl{y_1}\rd{x_8}\in\Shuf(8,1)$, we get $\tilde\sigma(\ubf)=(1,1,2,2,0,0,2,0,2)$.

\begin{proposition}\label{prop:bubble_hochschild}
	For $n>0$, the map $\tilde\sigma$ is an isomorphism from $\Bub(n-1,1)$ to $\Hoch(n)$.
\end{proposition}
\begin{proof}
	We first need to show that $\tilde\sigma$ is in fact a bijection from $\Shuf(n-1,1)$ to $\Tri(n)$.  Let $\ubf\in\Shuf(n-1,1)$ and let $\tilde\sigma(\ubf)=(u_{1},u_{2},\ldots,u_{n})$.  By construction, $u_{i}\in\{0,1,2\}$ which establishes \eqref{it:tri_1}.  If $u_{1}=2$, then it must be that $x_{n}\notin\ubf$, but $x_{n}$ is not in the alphabet used for constructing $\Shuf(n-1,1)$, so \eqref{it:tri_2} is satisfied.  If $y_{1}$ is not in the support of $\ubf$, then $u_{i}\in\{0,2\}$ for all $i\in[n]$.  Otherwise it appears directly after $x_{s}$, and we get that $u_{i}\in\{1,2\}$ for $i\leq n-s$ and $u_{i}\in\{0,2\}$ for $i>n-s$, which establishes \eqref{it:tri_3}.  Moreover, since 
	$\ubf$ is uniquely determined by its restriction to $\ubf_{\xbf}$ and by the position of the letter $y_{1}$ it is clear that the map $\tilde\sigma$ is injective.  It follows from \cite{combe:geometric}*{Eq.~(1.9)} and \cite{greene:shuffle}*{Theorem~3.4} that $\bigl\lvert\Tri(n)\bigr\rvert=2^{n-2}(n+3)=\bigl\lvert\Shuf(n-1,1)\bigr\rvert$, which implies that $\tilde\sigma$ is a bijection.
	
	\medskip
	
	It remains to show that $\ubf\bubcov\vbf$ if and only if $\tilde\sigma(\ubf)\compcov\tilde\sigma(\vbf)$ for all $\ubf,\vbf\in\Shuf(n-1,1)$.  Let $\ubf,\vbf\in\Shuf(n-1,1)$ and let $\tilde\sigma(\ubf)=(u_{1},u_{2},\ldots,u_{n})$, $\tilde\sigma(\vbf)=(v_{1},v_{2},\ldots,v_{n})$.  
	
	\medskip
	
	Suppose that $\ubf\bubcov\vbf$.  By Lemma~\ref{lem:bubble_covers}, we either have $\ubf\transpose\vbf$ or $\ubf\sindel\vbf$.  \\
	(i) If $\ubf\transpose\vbf$, then we may assume that $x_{s}x_{t}y_{1}$ is a consecutive subword of $\ubf$ and $x_{s}y_{1}x_{t}$ is a consecutive subword of $\vbf$.  In particular, $s<t$.  By construction, $\bigl\{i\in[n]\mid u_{i}=2\bigr\}=\bigl\{i\in[n]\mid v_{i}=2\bigr\}$, $u_{n+1-s}=0$ and $u_{i}\neq 0$ for all $i<n+1-s$.  Since $s<t$ it follows that $n+1-t<n+1-s$ which implies that $v_{n+1-s}=1$.  Now, choose $j\in[n]$ such that $u_{j}\neq 2$.  If $j\leq n-s$, then $u_{j}=1=v_{j}$, because $u_{i}=2=v_{i}$ for all $i\in\{n+2-t,\ldots,n-s\}$.  If $j>n+1-s$, then $u_{j}=0=v_{j}$ by construction.  It follows that $\tilde\sigma(\ubf)\compcov\tilde\sigma(\vbf)$.
	
	(ii) So assume that $\ubf\sindel\vbf$.  If $\vbf$ is obtained from $\ubf$ by adding $y_{1}$, then $\vbf=\ubf y_{1}$, which means that $u_{i}=v_{i}$ for all $i>1$ and $u_{1}=0$, $v_{1}=1$; thus $\tilde\sigma(\ubf)\compcov\tilde\sigma(\vbf)$.  If $\vbf$ is obtained from $\ubf$ by deleting the letter $x_{s}$, then either $\ubf$ does not contain $y_{1}$ or $\ubf$ has a consecutive subword $x_{s}y_{1}$.  In both cases, $u_{i}=v_{i}$ for $i\neq n+1-s$ and $u_{n+1-s}=0$ and $v_{n+1-s}=2$.  Since $u_{1}=0=v_{1}$, \eqref{it:tri_3} implies $\tilde\sigma(\ubf)\compcov\tilde\sigma(\vbf)$.  
	
	\medskip
	
	Conversely, suppose that $\tilde\sigma(\ubf)\compcov\tilde\sigma(\vbf)$.  By \cite{combe:geometric}*{Proposition~1.3}, there exists a unique index $i\in[n]$ such that $u_{i}<v_{i}$.\\
	(i) If $u_{i}=0$ and $v_{i}=1$, then $u_{j}\neq 0$ for all $j>i$ by \eqref{it:tri_3}, and consequently $v_{j}\neq 0$ for all $j>i$.  Moreover, $\bigl\{j\in[n]\mid u_{j}=2\bigr\}=\bigl\{j\in[n]\mid v_{j}=2\bigr\}$, and therefore $\ubf_{\xbf}=\vbf_{\xbf}$.  If $i=1$, then $\ubf$ does not contain the letter $y_{1}$, but $\vbf$ does.  However, since $v_{1}$ is the only entry equal to $1$, we conclude that $\vbf=\ubf y_{1}$, which means that $\ubf\sindel\vbf$ and thus $\ubf\bubcov\vbf$.  If $i>1$, then \eqref{it:tri_3} implies that $v_{1}=1$ and therefore $u_{1}=1$, meaning that $y_{1}$ is a letter of both $\ubf$ and $\vbf$.  Moreover, $x_{n+1-i}$ is also a letter of both $\ubf$ and $\vbf$.  Since $u_{i}$ is the first entry equal to $0$, it must be that $y_{1}$ comes directly after $x_{n+1-i}$ in $\ubf$.  If $y_{1}$ comes after $x_{n+1-i}$ in $\vbf$, say immediately after the letter $x_{n+1-s}$ for $s\leq i$, then by construction we would have $v_{s}=0$, and thus $v_{i}=0$ by \eqref{it:tri_3}.  It follows that $y_{1}$ comes before $x_{n+1-i}$ in $\vbf$.  Let $j$ be the smallest index $>i$ such that $u_{j}=0$ (if it exists).  Then $v_{j}=0$, and it follows that $\ubf$ has a consecutive subword $x_{n+1-j}x_{n+1-i}y_{1}$. If $y_{1}$ would come before $x_{n+1-j}$ in $\vbf$, say immediately after some letter $x_{n+1-s}$ for $j<s$, then $v_{s-1}=1$.  But now, \eqref{it:tri_3} forces $v_{j}=1$, a contradiction.  Thus, $x_{n+1-j}y_{1}x_{n+1-i}$ is a consecutive subword of $\vbf$, and we get $\ubf\transpose\vbf$.  If such a $j$ does not exist, then $\ubf$ starts with $x_{n+1-i}y_{1}$, while $\vbf$ starts with $y_{1}x_{n+1-i}$, and we still have $\ubf\transpose\vbf$.  Either way, we conclude $\ubf\bubcov\vbf$.
	
	(ii) If $u_{i}\neq 2$ and $v_{i}=2$, then $i>1$.  In particular, $x_{n+1-i}\notin\vbf$, but $x_{n+1-i}\in\ubf$.  If $u_{i}=1$, then $y_{1}$ appears in $\ubf$ in a position $x_{n-j+1}$ for $j\geq i$.  Thus, we remove $x_{n+1-i}$ from $\ubf$ from a position after $y_{1}$, which implies that this operation is a right indel and thus $\ubf\bubcov\vbf$.  If $u_{i}=0$, then we conclude that there exists some $j<i$ such that $u_{j}=0$.  Otherwise, the word $(v'_{1},v'_{2},\ldots,v'_{n})$ with $v'_{j}=u_{j}$ for all $j\neq i$ and $v'_{i}=u_{i+1}$ would be a triword, contradicting the assumption that $\tilde\sigma(\ubf)\compcov\tilde\sigma(\vbf)$.  It is either the case that $u_{j}\neq 1$ for all $j$ (and thus $v_{j}\neq 1$ for all $j$) which implies that $y_{1}\notin\ubf$ and $y_{1}\notin\vbf$.  But then, deleting $x_{n+1-i}$ from $\ubf$ clearly is a right indel, and we get $\ubf\bubcov\vbf$.  Otherwise, we choose the largest index $j<i$ with $u_{j}=0$ (whose existence we have argued above).  By construction $y_{1}$ comes right after $x_{n+1-j}$ which comes after $x_{n+1-i}$ in $\ubf$.  Once again, removing $x_{n+1-i}$ from $\ubf$ is a right indel, and we get $\ubf\bubcov\vbf$.

	The proof is thus complete.
\end{proof}

Figure~\ref{fig:hochschild_map_3} shows the action of the map $\tilde\sigma$ for $n=3$, and Figure~\ref{fig:hochschild_3} shows the lattice $\Hoch(3)$.  

\begin{figure}
	\centering
	\begin{subfigure}[b]{.45\textwidth}
		\centering
		\begin{tabular}{r|l}
			$\ubf\in\Shuf(2,1)$ & $\tilde\sigma(\ubf)$\\
			\hline\hline
			$x_{1}x_{2}$ & $(0,0,0)$\\
			$x_{1}$ & $(0,2,0)$ \\
			$x_{2}$ & $(0,0,2)$\\
			$\epsilon$ & $(0,2,2)$\\
			$y_{1}x_{1}$ & $(1,2,1)$\\
			$x_{1}y_{1}$ & $(1,2,0)$\\
			$y_{1}x_{2}$ & $(1,1,2)$\\
			$x_{2}y_{1}$ & $(1,0,2)$\\
			$x_{1}x_{2}y_{1}$ & $(1,0,0)$\\
			$x_{1}y_{1}x_{2}$ & $(1,1,0)$\\
			$y_{1}x_{1}x_{2}$ & $(1,1,1)$\\
		\end{tabular}
		\caption{The action of $\tilde\sigma$ for $n=3$.}
		\label{fig:hochschild_map_3}
	\end{subfigure}
	\hspace*{1cm}
	\begin{subfigure}[b]{.45\textwidth}
		\centering
		\includegraphics[page=8,scale=1]{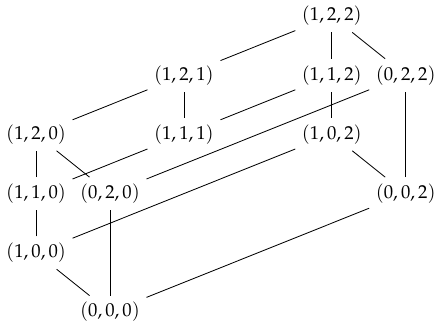}
		\caption{The lattice $\Hoch(3)$.}
		\label{fig:hochschild_3}
	\end{subfigure}
	\caption{Illustration of Proposition~\ref{prop:bubble_hochschild}.}
	\label{fig:hochschild}
\end{figure}

\begin{remark}
	The map $\tilde\sigma$ is slightly different from the map $\sigma\colon\Tri(n)\to\Shuf(n-1,1)$ used in \cite{muehle:hochschild}*{Section~5.3}.  For $\ubf\in\Shuf(n-1,1)$, we get $\sigma^{-1}(\ubf)$ by reversing $\tilde\sigma(\ubf)$ and reversing the letters of $\xbf$ once again.
\end{remark}

\begin{question}\label{qu:vector_realization}
	Proposition~\ref{prop:bubble_hochschild} explains how to realize $\Bub(n-1,1)$ as the componentwise order on certain integer tuples.  Can we extend this construction to $\Bub(m,n)$ for arbitrary $m,n\geq 0$, \ie can we find a set $X_{m,n}\subseteq\mathbb{Z}^{d}$ of integer tuples such that $\Bub(m,n)\cong(X_{m,n},\compleq)$?  What is the minimum value for $d$?
\end{question}

In fact, an affirmative answer to Question~\ref{qu:vector_realization} would have consequences for computing the order dimension of $\Bub(m,n)$.  Recall that the \defn{order dimension} of a finite poset ${\Poset=(P,\leq)}$ is the least number $d$ such that the order relation $\leq$ arises as the intersection of $d$ linear extensions of $\leq$.  Equivalently, the order dimension of $\Poset$ is the least number $d$ such that $\Poset$ can be embedded into a direct product of $d$ chains~\cite{hiraguti:dimension}*{Theorem~9.6}.  Thus, if we have $\Poset\cong(Q,\compleq)$ for some $Q\subseteq\mathbb{Z}^{d}$, then $\dim\Poset\leq d$.

The \defn{$k$-crown} is the poset $\mathbf{C}_{k}$ on the ground set $\{p_{1},p_{2},\ldots,p_{k},q_{1},q_{2},\ldots,q_{k}\}$ determined by the covering pairs $p_{i}\lessdot q_{j}$ if and only if $i\neq j$.  It is well known that $\dim\mathbf{C}_{k}=k$~\cite{hiraguti:dimension}*{Theorem~7.3}.  If $\Poset$ is a lattice, then the elements covering the bottom element are called \defn{atoms}.

\begin{lemma}\label{lem:semidistributive_odim}
    A finite semidistributive lattice with $k$ atoms has order dimension at least $k$.
\end{lemma}
\begin{proof}
    Let $\Lattice=(L,\leq)$ be semidistributive and let $A$ denote its set of atoms.  Since $A\subseteq\JI(\Lattice)$, \cite{freese:free}*{Theorem~2.56} implies that for every $a\in A$, the set
    \begin{displaymath}
        K(a) \defs \bigl\{p\in L\mid a\not\leq p\bigr\}
    \end{displaymath}
    has a greatest element, denoted by $\kappa(a)$.  Moreover, \cite{freese:free}*{Corollary~2.55} asserts that the assignment $a\mapsto\kappa(a)$ is injective.  So, if $a'\in A\setminus\{a\}$, then we necessarily have $a'\in K(a)$, because the atoms of $\Lattice$ form an antichain.  It follows that $a'<\kappa(a)$ for all $a,a'\in A$ with $a\neq a'$.  We also conclude that $\bigl\{\kappa(a)\mid a\in A\bigr\}$ is an antichain, because otherwise there would exist distinct atoms $a,a'\in A$ with $\kappa(a)\leq\kappa(a')$.  But then, $a'<\kappa(a)\leq\kappa(a')$, which contradicts $\kappa(a')\in K(a')$.  Therefore, the set
    \begin{displaymath}
        \Bigl(A\cup\bigl\{\kappa(a)\mid a\in A\bigr\},\leq\Bigr)
    \end{displaymath}
    exhibits a $k$-crown as a subposet of $\Lattice$.  By definition, the dimension of a subposet of $\Lattice$ cannot exceed the dimension of $\Lattice$.
\end{proof}

\begin{corollary}\label{cor:bubble_odim}
    For $m,n\geq 0$, we have $\dim\Bub(m,n)\geq m+n$.
\end{corollary}
\begin{proof}
    Lemma~\ref{lem:bubble_irreducibles} implies that $\Bub(m,n)$ has exactly $m+n$ atoms, namely $\xbf_{\hat{\imath}}$ for $i\in[m]$ and $\xbf y_{j}$ for $j\in[n]$.  The claim follows from Lemma~\ref{lem:semidistributive_odim}.
\end{proof}

We conjecture that the bound given in Corollary~\ref{cor:bubble_odim} is sharp.  This can, for instance, be proven by giving an affirmative answer to Question~\ref{qu:vector_realization} by realizing $\Bub(m,n)$ as the componentwise order on certain integer tuples of length $m+n$. For $n=1$, this was achieved in Proposition~\ref{prop:bubble_hochschild}.

\begin{conjecture}
    For $m,n\geq 0$, we have $\dim\Bub(m,n)=m+n$.
\end{conjecture}

\subsection{The Galois Graph of $\Bub(m,n)$}
	\label{sec:bubble_galois}
Any extremal lattice admits a neat representation in terms of certain set pairs coming from a directed graph.  More precisely, let $\Lattice=(L,\leq)$ be a finite extremal lattice of length $k$.  We may choose any maximal chain of length $k$ in $\Lattice$, say $C=\{c_{0},c_{1},\ldots,c_{k}\}$ to order the join- and meet-irreducible elements of $\Lattice$ as $\JI(\Lattice)=\{j_{1},j_{2},\ldots,j_{k}\}$ and $\MI(\Lattice)=\{m_{1},m_{2},\ldots,m_{k}\}$ such that
\begin{equation}\label{eq:extremal_order}
	j_{1}\vee j_{2}\vee\cdots\vee j_{s} = c_{s} = m_{s+1}\wedge m_{s+2}\wedge\cdots\wedge m_{k}
\end{equation}
for all $s\in[k]$.  If $s=k$, then the right side of \eqref{eq:extremal_order} is the meet over the empty set, which by default is the top element $c_{k}$ of $\Lattice$.  The \defn{Galois graph} of $\Lattice$ is the directed graph $\Galois(\Lattice)\defs\bigl([k],\leadsto\bigr)$, where $s\leadsto t$ if and only if $s\neq t$ and $j_{s}\not\leq m_{t}$.

Given any directed graph $(V,E)$, an \defn{orthogonal pair} is a pair $(A,B)$, where $A,B\subseteq V$, $A\cap B=\emptyset$ and $(A\times B)\cap E=\emptyset$.  In other words, an orthogonal pair is a pair of disjoint sets of vertices where no arrow goes from the first component to the second.  An orthogonal pair $(A,B)$ is \defn{maximal} if adding any element to $A$ or $B$ violates the orthogonality property.  

\begin{theorem}[\cite{markowsky:primes}*{Theorem~11}]\label{thm:galois_graph}
	If $\Lattice$ is a finite, extremal lattice, then $\Lattice$ is isomorphic to the lattice of maximal orthogonal pairs of $\Galois(\Lattice)$ ordered by inclusion of first components.
\end{theorem}

Note that, a priori, the Galois graph depends on a choice of maximal chain.  Theorem~\ref{thm:galois_graph}, however, suggests that this is not the case.  Indeed, \cite{thomas:rowmotion}*{Proposition~2.5} states that different maximal chains of length $k$ yield isomorphic Galois graphs.

The next result reduces the amount of work constructing the Galois graph when $\Lattice$ is also semidistributive.

\begin{lemma}[\cite{muehle:noncrossing}*{Corollary~A.18(ii)}]\label{lem:galois_semidistributive}
	Let $\Lattice$ be a finite, semidistributive extremal lattice of length $k$. Suppose that $\JI(\Lattice)$ and $\MI(\Lattice)$ are ordered as in \eqref{eq:extremal_order} with respect to some maximal chain of length $k$.  For $s,t\in[k]$, it holds that $j_{s}\not\leq m_{t}$ if and only if $s\neq t$ and $j_{t}\leq{j_{t}}_{*}\vee j_{s}$.
\end{lemma}

In fact, \cite{muehle:noncrossing}*{Corollary~A.18} is stated for interval-constructable lattices, but its proof uses only properties of semidistributive lattices and extends verbatim to the more general setting stated here.  As a consequence, we can compute the Galois graph of a finite semidistributive extremal lattice using only its join-irreducible elements, while the original construction uses both sets of irreducible elements.

We now describe the Galois graph of $\Bub(m,n)$.  Recall that $\Tcal=X\uplus Y\uplus(X\times Y)$ is the vertex set of the noncrossing matching complex $\Gamma_{m,n}$.

\begin{proposition}\label{prop:bubble_galois}
	For $m,n\geq 0$, the Galois graph of $\Bub(m,n)$ is the directed graph $(\Tcal,E)$, where $(\ell_{1},\ell_{2})\in E$ if and only if $\ell_{1}\neq\ell_{2}$ and either
	\begin{itemize}
		\item $\ell_{1}=(x_{s},y_{t})$ and $\ell_{2}=x_{s}$,\quad or
		\item $\ell_{1}=y_{t}$ and $\ell_{2}=(x_{s},y_{t})$,\quad or
		\item $\ell_{1}=(x_{s},y_{t})$ and $\ell_{2}=(x_{s'},y_{t'})$ with $s\geq s'$ and $t\leq t'$.
	\end{itemize}
\end{proposition}
\begin{proof}
	By Lemma~\ref{lem:bubble_irreducibles}, there are essentially two types of join-irreducible elements.  We define
	\begin{align*}
		\abf^{(s,j)} & \defs x_{1}x_{2}\cdots x_{s}y_{j}x_{s+1}\cdots x_{m}, && \text{for}\;0\leq s\leq m,\;1\leq j\leq n\\
		\bbf^{(s)} & \defs x_{1}x_{2}\cdots\hat{x_{s}}\cdots x_{n}, && \text{for}\;1\leq s\leq n.
	\end{align*}
	It is immediately clear that
	\begin{displaymath}
		\bbf^{(s)}_{*} = \xbf \quad \text{and} \quad \abf^{(s,j)}_{*} = \begin{cases}\abf^{(s+1,j)}, & \text{if}\;s<m,\\\xbf, & \text{if}\;s=m.\end{cases}
	\end{displaymath}
	and $\abf^{(s,j)}<\abf^{(t,j)}$ whenever $s>t$.  See also Corollary~\ref{cor:bubble_irreducibles_poset}.

	Now, consider two join-irreducibles $\ubf$ and $\vbf$.  We want to understand the cases, when $\ubf\bubleq\ubf_{*}\vee\vbf$.
	
	(i) If $\ubf=\bbf^{(i)}$, then $\ubf_{*}\vee\vbf=\vbf$.  Corollary~\ref{cor:bubble_irreducibles_poset} now tells us that $\ubf\bubleq\vbf$ if and only if $\ubf=\vbf$.

	(ii) If $\ubf=\abf^{(m,i)}$, then $\ubf_{*}\vee\vbf=\vbf$.  Corollary~\ref{cor:bubble_irreducibles_poset} now tells us that $\ubf\bubleq\vbf$ if and only if $\vbf=\abf^{(s,i)}$.

	(iii) If $\ubf=\abf^{(s,i)}$ for $s<m$, then $\ubf_{*}=\abf^{(s+1,i)}$.  Consequently, $(x_{s+1},y_{i})\in\invset(\ubf)\setminus\invset(\ubf_{*})$.  Let $\wbf=\ubf_{*}\vee\vbf$.

	(iiia) Suppose first that $\vbf=\abf^{(t,j)}$.  Lemma~\ref{lem:bubble_joins_fill} then tells us that that $\vbf$ contributes an inversion $(x_{t+1},y_{j})$ to $\wbf$ and $\ubf_{*}$ contributes an inversion $(x_{s+2},y_{i})$ to $\wbf$.  If $i\leq j$, then we also have an inversion $(x_{t+1},y_{i})$ in $\wbf$.  For $\ubf\bubleq\wbf$ to hold it must necessarily be that $(x_{s+1},y_{i})$ is an inversion of $\wbf$, which can only be achieved if $t\leq s$.  If $i>j$, then $y_{i}$ comes after $y_{j}$ in $\wbf$ and there is no way that $(x_{s+1},y_{i})$ is an inversion of $\wbf$.

    (iiib) Now suppose that $\vbf=\bbf^{(j)}$.  In view of Lemma~\ref{lem:bubble_joins_fill}, we get that $\wbf_{\xbf}=X\setminus\{x_{j}\}$.  In particular, $\wbf$ does not have more inversions than $\ubf_{*}$.  This means that we have $\ubf\not\bubleq\wbf$ except when $j=s+1$, because then the desired inversion of $(x_{s+1},y_{i})$ is not relevant as $x_{s+1}$ is not in the support of $\wbf$.  Thus, we have have $\ubf\bubleq\wbf$ if and only if $j=s+1$.

    \medskip

    In summary, we have the following relations among the join-irreducible elements of $\Bub(m,n)$:
    \begin{itemize}
        \item no arrows leave $\bbf^{(s)}$ and an arrow enters $\bbf^{(s)}$ if and only if it comes from $\abf^{(s-1,j)}$;
        \item no arrow enters $\abf^{(m,i)}$ and $\abf^{(m,i)}$ has an arrow to every $\abf^{(s,i)}$ for $s<m$;
        \item there is an arrow from $\abf^{(s,i)}$ to $\abf^{(t,j)}$ if and only if $i\leq j$ and $s\geq t$, except when $i=j$ and $s=t$.
    \end{itemize}
    In view of \eqref{eq:bubble_labeling}, we get
    \begin{displaymath}
        \lambda(\ubf_{*},\ubf) = \begin{cases}x_{s} , & \text{if}\;\ubf=\bbf^{(s)},\\y_{i}, & \text{if}\;\ubf=\abf^{(m,i)},\\(x_{s+1},y_{i}) , & \textbf{if}\;\ubf=\abf^{(s,i)},s<m.\end{cases}
    \end{displaymath}
    This gives the encoding in the statement.
\end{proof}

\begin{figure}
	\centering
	\begin{subfigure}[t]{.45\textwidth}
		\centering
		\includegraphics[page=11,scale=1]{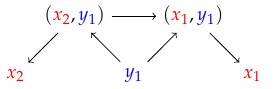}
		\caption{The Galois graph of $\Bub(2,1)$.}
		\label{fig:bubble_galois_21}
	\end{subfigure}
	\hspace*{1cm}
	\begin{subfigure}[t]{.45\textwidth}
		\centering
		\includegraphics[page=12,scale=1]{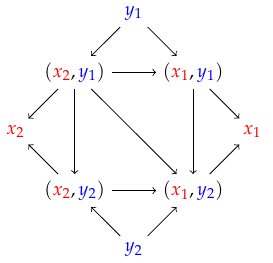}
		\caption{The Galois graph of $\Bub(2,2)$.}
		\label{fig:bubble_galois_22}
	\end{subfigure}
	\caption{Some Galois graphs of bubble lattices.}
	\label{fig:bubble_galois}
\end{figure}

%

\bibliographystyle{plain}
\bibliography{bib_shuffle}

\begin{bibdiv}
\begin{biblist}

\bib{adaricheva:join}{article}{
      author={Adaricheva, Kira~V.},
      author={Gorbunov, Viktor~A.},
      author={Tumanov, V.~I.},
       title={Join-semidistributive lattices and convex geometries},
        date={2003},
        ISSN={0001-8708},
     journal={Adv. Math.},
      volume={173},
      number={1},
       pages={1\ndash 49},
         url={https://doi.org/10.1016/S0001-8708(02)00011-7},
      review={\MR{1954454}},
}

\bib{barnard:canonical}{article}{
      author={Barnard, Emily},
       title={The canonical join complex of the {T}amari lattice},
        date={2020},
        ISSN={0097-3165},
     journal={J. Combin. Theory Ser. A},
      volume={174},
       pages={105207, 30},
         url={https://doi.org/10.1016/j.jcta.2019.105207},
      review={\MR{4082060}},
}

\bib{chapoton:dyck}{article}{
      author={Chapoton, Fr\'{e}d\'{e}ric},
       title={Some properties of a new partial order on {D}yck paths},
        date={2020},
     journal={Algebr. Comb.},
      volume={3},
      number={2},
       pages={433\ndash 463},
         url={https://doi.org/10.5802/alco.98},
      review={\MR{4099002}},
}

\bib{combe:geometric}{article}{
      author={Combe, Camille},
       title={A geometric and combinatorial exploration of {H}ochschild
  lattices},
        date={2021},
     journal={Electron. J. Combin.},
      volume={28},
      number={2},
       pages={Paper No. 2.38, 29},
      review={\MR{4281204}},
}

\bib{day:characterizations}{article}{
      author={Day, Alan},
       title={Characterizations of finite lattices that are bounded-homomorphic
  images of sublattices of free lattices},
        date={1979},
        ISSN={0008-414X},
     journal={Canadian J. Math.},
      volume={31},
      number={1},
       pages={69\ndash 78},
         url={https://doi.org/10.4153/CJM-1979-008-x},
      review={\MR{518707}},
}

\bib{day:doubling}{article}{
      author={Day, Alan},
       title={Doubling constructions in lattice theory},
        date={1992},
        ISSN={0008-414X},
     journal={Canad. J. Math.},
      volume={44},
      number={2},
       pages={252\ndash 269},
         url={https://doi.org/10.4153/CJM-1992-017-7},
      review={\MR{1162342}},
}

\bib{day:congruence}{article}{
      author={Day, Alan},
       title={Congruence normality: the characterization of the doubling class
  of convex sets},
        date={1994},
        ISSN={0002-5240},
     journal={Algebra Universalis},
      volume={31},
      number={3},
       pages={397\ndash 406},
         url={https://doi.org/10.1007/BF01221793},
      review={\MR{1265350}},
}

\bib{doran:shuffling}{article}{
      author={Doran, William~F., IV},
       title={Shuffling lattices},
        date={1994},
        ISSN={0097-3165},
     journal={J. Combin. Theory Ser. A},
      volume={66},
      number={1},
       pages={118\ndash 136},
         url={https://doi.org/10.1016/0097-3165(94)90054-X},
      review={\MR{1273295}},
}

\bib{freese:free}{book}{
      author={Freese, Ralph},
      author={Je\v{z}ek, Jaroslav},
      author={Nation, J.~B.},
       title={Free lattices},
      series={Mathematical Surveys and Monographs},
   publisher={American Mathematical Society, Providence, RI},
        date={1995},
      volume={42},
        ISBN={0-8218-0389-1},
         url={https://doi.org/10.1090/surv/042},
      review={\MR{1319815}},
}

\bib{garver.mcconville:oriented_tree}{article}{
      author={Garver, Alexander},
      author={McConville, Thomas},
       title={Oriented flip graphs of polygonal subdivisions and noncrossing
  tree partitions},
        date={2018},
        ISSN={0097-3165},
     journal={J. Combin. Theory Ser. A},
      volume={158},
       pages={126\ndash 175},
         url={https://doi.org/10.1016/j.jcta.2018.03.014},
      review={\MR{3800125}},
}

\bib{greene:shuffle}{article}{
      author={Greene, Curtis},
       title={Posets of shuffles},
        date={1988},
        ISSN={0097-3165},
     journal={J. Combin. Theory Ser. A},
      volume={47},
      number={2},
       pages={191\ndash 206},
         url={https://doi.org/10.1016/0097-3165(88)90018-0},
      review={\MR{930953}},
}

\bib{hersh:shuffle}{article}{
      author={Hersh, Patricia},
       title={Two generalizations of posets of shuffles},
        date={2002},
        ISSN={0097-3165},
     journal={J. Combin. Theory Ser. A},
      volume={97},
      number={1},
       pages={1\ndash 26},
         url={https://doi.org/10.1006/jcta.2001.3187},
      review={\MR{1879043}},
}

\bib{hiraguti:dimension}{article}{
      author={Hiraguti, Tosio},
       title={On the dimension of orders},
        date={1955},
        ISSN={0022-8338},
     journal={Sci. Rep. Kanazawa Univ.},
      volume={4},
      number={1},
       pages={1\ndash 20},
      review={\MR{77500}},
}

\bib{markowsky:primes}{article}{
      author={Markowsky, George},
       title={Primes, irreducibles and extremal lattices},
        date={1992},
        ISSN={0167-8094},
     journal={Order},
      volume={9},
      number={3},
       pages={265\ndash 290},
         url={https://doi.org/10.1007/BF00383950},
      review={\MR{1211380}},
}

\bib{markowsky:permutation}{article}{
      author={Markowsky, George},
       title={Permutation lattices revisited},
        date={1994},
        ISSN={0165-4896},
     journal={Math. Social Sci.},
      volume={27},
      number={1},
       pages={59\ndash 72},
         url={https://doi.org/10.1016/0165-4896(94)00731-4},
      review={\MR{1267688}},
}

\bib{mcconville.muehle:bubbleII}{article}{
      author={McConville, Thomas},
      author={M{\"u}hle, Henri},
       title={Bubble lattices {II}: Combinatorics},
        date={2022},
        note={In preparation.},
}

\bib{muehle:core}{article}{
      author={M\"{u}hle, Henri},
       title={The core label order of a congruence-uniform lattice},
        date={2019},
        ISSN={0002-5240},
     journal={Algebra Universalis},
      volume={80},
      number={1},
       pages={Paper No. 10, 22},
         url={https://doi.org/10.1007/s00012-019-0585-5},
      review={\MR{3908324}},
}

\bib{muehle:hochschild}{article}{
      author={M{\"u}hle, Henri},
       title={Hochschild lattices and shuffle lattices},
        date={2020},
      eprint={arXiv:2008.13247},
}

\bib{muehle:extremality}{article}{
      author={M{\"u}hle, Henri},
       title={Extremality, left-modularity and semidistributivity},
        date={2021},
      eprint={arXiv:2112.07959},
}

\bib{muehle:distributive}{article}{
      author={M{\"u}hle, Henri},
       title={Meet-distributive lattices have the intersection property},
        date={2021},
      eprint={arXiv:1810.01528v4},
}

\bib{muehle:noncrossing}{article}{
      author={M\"{u}hle, Henri},
       title={Noncrossing arc diagrams, {T}amari lattices, and parabolic
  quotients of the symmetric group},
        date={2021},
        ISSN={0218-0006},
     journal={Ann. Comb.},
      volume={25},
      number={2},
       pages={307\ndash 344},
         url={https://doi.org/10.1007/s00026-021-00532-9},
      review={\MR{4268292}},
}

\bib{reading:lattice}{article}{
      author={Reading, Nathan},
       title={Lattice and order properties of the poset of regions in a
  hyperplane arrangement},
        date={2003},
        ISSN={0002-5240},
     journal={Algebra Universalis},
      volume={50},
      number={2},
       pages={179\ndash 205},
         url={https://doi.org/10.1007/s00012-003-1834-0},
      review={\MR{2037526}},
}

\bib{reading:clusters}{article}{
      author={Reading, Nathan},
       title={Clusters, {C}oxeter-sortable elements and noncrossing
  partitions},
        date={2007},
        ISSN={0002-9947},
     journal={Trans. Amer. Math. Soc.},
      volume={359},
      number={12},
       pages={5931\ndash 5958},
         url={https://doi.org/10.1090/S0002-9947-07-04319-X},
      review={\MR{2336311}},
}

\bib{reading:shard}{article}{
      author={Reading, Nathan},
       title={Noncrossing partitions and the shard intersection order},
        date={2011},
        ISSN={0925-9899},
     journal={J. Algebraic Combin.},
      volume={33},
      number={4},
       pages={483\ndash 530},
         url={https://doi.org/10.1007/s10801-010-0255-3},
      review={\MR{2781960}},
}

\bib{reading:lattice_theory}{incollection}{
      author={Reading, Nathan},
       title={Lattice theory of the poset of regions},
        date={2016},
   booktitle={Lattice theory: special topics and applications. {V}ol. 2},
   publisher={Birkh\"{a}user/Springer, Cham},
       pages={399\ndash 487},
      review={\MR{3645055}},
}

\bib{simion_stanley:shuffle}{article}{
      author={Simion, Rodica},
      author={Stanley, Richard~P.},
       title={Flag-symmetry of the poset of shuffles and a local action of the
  symmetric group},
        date={1999},
        ISSN={0012-365X},
     journal={Discrete Math.},
      volume={204},
      number={1-3},
       pages={369\ndash 396},
         url={https://doi.org/10.1016/S0012-365X(98)00381-1},
      review={\MR{1691879}},
}

\bib{stembridge:fully}{article}{
      author={Stembridge, John~R.},
       title={On the fully commutative elements of {C}oxeter groups},
        date={1996},
        ISSN={0925-9899},
     journal={J. Algebraic Combin.},
      volume={5},
      number={4},
       pages={353\ndash 385},
         url={https://doi.org/10.1023/A:1022452717148},
      review={\MR{1406459}},
}

\bib{thomas:analogue}{article}{
      author={Thomas, Hugh},
       title={An analogue of distributivity for ungraded lattices},
        date={2006},
        ISSN={0167-8094},
     journal={Order},
      volume={23},
      number={2-3},
       pages={249\ndash 269},
         url={https://doi.org/10.1007/s11083-006-9046-9},
      review={\MR{2308910}},
}

\bib{thomas:rowmotion}{article}{
      author={Thomas, Hugh},
      author={Williams, Nathan},
       title={Rowmotion in slow motion},
        date={2019},
        ISSN={0024-6115},
     journal={Proc. Lond. Math. Soc. (3)},
      volume={119},
      number={5},
       pages={1149\ndash 1178},
         url={https://doi.org/10.1112/plms.12251},
      review={\MR{3968720}},
}

\end{biblist}
\end{bibdiv}

\end{document}